\documentclass[reqno]{amsart}

\usepackage{amssymb,a4wide}
\usepackage{xcolor}
\definecolor{pgcopypositive}{RGB}{33,90,160}
\definecolor{pgcopynegative}{RGB}{145,75,15}

\usepackage{etoolbox}
\AtBeginEnvironment{quote}{\normalfont\normalsize\itshape}

\usepackage[colorlinks=true,linkcolor=blue,citecolor=blue,urlcolor=blue]{hyperref}
\newtheorem{theorem}{Theorem}
\newtheorem{lemma}[theorem]{Lemma}
\newtheorem{corollary}[theorem]{Corollary}
\newtheorem{proposition}[theorem]{Proposition}
\newtheorem{maintheorem}{Theorem}

\theoremstyle{remark}
\newtheorem{remark}[theorem]{Remark}
\newtheorem*{example}{Example}
\newcommand{\E}{\mathbb E}
\newcommand{\K}{\mathbb K}
\newcommand{\N}{\mathbb N}
\newcommand{\diag}{\operatorname{diag}}

\newcommand{\norm}[1]{\left\lVert#1\right\rVert}
\newcommand{\abs}[1]{\left\lvert#1\right\rvert}
\title[Uniqueness of symmetric structure]{On Pe\l czy\'nski's problem concerning the uniqueness of symmetric structure}
\author{Tianbao Guo}
\address{Department of Mathematics, Harbin University of Science and Technology,
Harbin, 150080, China}
\email{tianbaoguo.edu@gmail.com}
\author{Jinghao Huang}
\address{Institute for Advanced Study in Mathematics of HIT, Harbin 150001, China}
\email{jinghao.huang@hit.edu.cn}
\thanks{Corresponding author: Jinghao Huang.}
\date{}
\begin{document}
\begin{abstract}
Let $E$ and $F$ be separable symmetric Banach sequence spaces, and let
$C_E$ and $C_F$ be the corresponding ideals of compact operators
on a separable Hilbert space.
We prove that if $C_F$ is isomorphic to a complemented subspace of $S_E$,
then $F=\ell_2$. As a consequence, if $C_E$ is isomorphic to $C_F$,
then $E=F$ with equivalent norms. This answers Pe\l czy\'nski's
1979 question concerning the uniqueness of symmetric structure.
We also strengthen Arazy's result by showing that if $E$ has finite
cotype and $C_F$ embeds isomorphically into $S_E$, then $F=\ell_2$.
For reflexive $F\ne\ell_2$ with non-trivial Boyd indices, the embedding
theorems of Lindenstrauss and Szankowski yield a separable reflexive
symmetric sequence space $E$ with trivial cotype such that $C_F$ embeds
isomorphically into $S_E$. This gives negative answers to several
questions raised by Arazy in 1981.
\end{abstract}
\maketitle

\section{Introduction}
\subsection{Uniqueness of symmetric structure}
It has long been known that isomorphic symmetric sequence spaces
need not have equivalent canonical bases; see \cite[Section 3.b]{LT}.
This leads to the question of whether symmetric ideals of compact
operators have a unique symmetric structure
(Pe\l czy\'nski, International Conference on Banach Space Theory
and its Applications, Kent, Ohio, August 1979 \cite{Arazy,Arazy83}):
\begin{quote}
Let $E$ and $F$ be separable symmetric Banach sequence spaces. If $C_E$ and $C_F$ are isomorphic, must $E$ and $F$ coincide as
sets of sequences with equivalent norms?
\end{quote}

Let $E$ be a symmetric Banach sequence space. The associated ideal
of compact operators on $\ell_2$ is
\[
 C_E=\{x\in\mathcal K(\ell_2):s(x)\in E\},\qquad
 \norm{x}_{C_E}=\norm{s(x)}_E,
\]
where $s(x)$ is the decreasing sequence of singular values of $x$.
We say that $C_E$ has a unique symmetric structure if, for every
separable symmetric sequence space $F$, an isomorphism $C_E\simeq C_F$
implies that $E$ and $F$ coincide as sets of sequences and have
equivalent norms. We denote this relation by $E=F$.

We write $\K$ for the scalar field.
We write $c_{00}$ for the space
of finitely supported sequences over $\K$.

The Schatten classes $C_p$, $1\le p<\infty$, and
$\mathcal K(\ell_2)$ with the operator norm were already known to have
unique symmetric structure; see \cite{Arazy83,AL75}.
Arazy studied this problem in \cite{Arazy,Arazy83} by considering
the subspace
\[
 S_E=\Bigl(\bigoplus_{n\geq1}M_n(\K)\Bigr)_{C_E},
\]
where $M_n(\K)$ acts on the $n$th summand of
$\ell_2=\bigoplus_{n\ge1}\K^n$, and each matrix algebra carries the
usual matrix trace.

Arazy proved the following alternative under the assumption that
neither $E$ nor $F$ contains
a subspace isomorphic to $c_0$: a complemented embedding of $C_F$
into $C_E$ implies that $F=\ell_2$, or $E=F$, or $C_F$ embeds
complementably into $S_E$ \cite[Theorem 5.6]{Arazy}.
For symmetric sequence spaces $E$ and $F$ containing no copy of
$c_0$, Arazy \cite[Corollary 5.7]{Arazy} also proved that
\[
 C_E\simeq C_F\quad\Longrightarrow\quad
 E=F\ \text{ or }\ C_E\simeq S_E\simeq S_F\simeq C_F.
\]
Moreover, if $E$ contains no copy of $c_0$ and $C_E\not\simeq S_E$,
then $C_E$ has unique symmetric structure
\cite[Corollary 5.8]{Arazy}.
Arazy \cite{Arazy83} also recorded that $C_E\simeq C_F$
implies $E\simeq F$ as Banach spaces. The converse fails
\cite[Theorem 4.14]{Arazy78}.

Huang and Sukochev \cite[Corollary 5.1]{HS2024} proved
uniqueness of symmetric structure under surjective linear isometries
between complex symmetric ideals with order-continuous norms.

We prove the following theorem.

\begin{maintheorem}\label{pg:uniqueness}
Let $E$ and $F$ be separable symmetric Banach sequence spaces.
If $C_E\simeq C_F$, then
$E=F$ as sets of sequences with equivalent norms.
\end{maintheorem}

The proof uses the implication
\begin{equation}\label{eq:intro-complemented}
 C_F\text{ embeds complementably into }S_E
 \quad\Longrightarrow\quad F=\ell_2
\end{equation}
for arbitrary separable symmetric $E$ and $F$
(Theorem~\ref{pg:comp-block}).
A coordinate selection and an operator-norm perturbation yield an
embedding $V$ and a bounded left inverse $R$ satisfying
\[
 Q_nV=V\Pi_n,\qquad \Pi_nR=RQ_n\qquad(n\geq1).
\]
Here $\Pi_n=P_n-P_{n-1}$ are the projections defined in
\eqref{pg:shells}. For a suitable sequence of integers
$0=N_0<N_1<N_2<\cdots$, define
\[
 Q_n:S_E\to S_E,\qquad
 Q_n(y)=\sum_{j=N_{n-1}+1}^{N_n}q_jyq_j\quad(n\geq1),
\]
where $q_j$ is the orthogonal projection onto the $j$th summand
$\K^j$ in $\ell_2=\bigoplus_{j\geq1}\K^j$.
Theorem~\ref{pg:pure-hilbert} gives
\[
 (3\norm{V}\norm{R})^{-1}\norm{a}_2
 \leq\norm{a}_F\leq3\norm{V}\norm{R}\,\norm{a}_2
 \qquad(a\in c_{00}).
\]
The estimate follows from Lemma~\ref{pg:mixed-inequality} and
averaging over invariant orthonormal pairs.
Theorem~\ref{pg:uniqueness} is then deduced from
Theorem~\ref{pg:comp-block}, Proposition~\ref{pg:bounded-alternative}
and Theorem~\ref{pg:unbounded-uniqueness}.

\subsection{Isomorphic embeddings}
Arazy's question on isomorphic embeddings, restated in
\cite[Question 1.1]{HSS}, asks whether
\begin{equation}\label{eq:intro-arazy}
 C_F\hookrightarrow S_E\quad\Longrightarrow\quad F=\ell_2
\end{equation}
hold for separable symmetric $E,F$? Here $\hookrightarrow$ denotes an
isomorphic embedding; the range need not be complemented.
Huang, Sadovskaya and Sukochev proved that no such embedding exists
if $E$ contains no copy of $\ell_2$ \cite[Theorem 4.1]{HSS}.
They proved \eqref{eq:intro-arazy} under the following condition:
every normalized block basic sequence of the canonical basis of $E$
whose $\ell_\infty$-norms tend to zero has a subsequence equivalent
to the unit vector basis of $\ell_2$ \cite[Theorem 4.4]{HSS}.
Following \cite{HSS}, we call this condition Property $(*)$.

The unrestricted implication \eqref{eq:intro-arazy} is false.
For every separable reflexive symmetric sequence space $F\ne\ell_2$
with non-trivial Boyd indices, classical embedding theorems give a
separable reflexive symmetric sequence space $E$ such that
\[
 C_F\hookrightarrow E\hookrightarrow S_E.
\]
We use the exponent convention of \cite{Arazy78}: non-trivial Boyd
indices mean $1<p_F\leq q_F<\infty$. This includes $F=\ell_p$,
$1<p<\infty$, $p\ne2$, with $C_p=C_{\ell_p}$.
The non-trivial Boyd indices imply that $C_F$ has an unconditional
finite-dimensional decomposition \cite[Corollary 4.6]{Arazy78}.
Lindenstrauss's embedding theorem \cite[Theorem 1.g.5]{LT}, together
with \cite[Theorem 3.1]{FJT}, gives an embedding of $C_F$ into a
reflexive space with an unconditional basis. Szankowski's
theorem then gives an embedding into a reflexive space with a
symmetric basis \cite{Szankowski};
see also \cite[Remark 3.2 and Theorem 3.3]{DKM}.
Theorem~\ref{thm:counterexample} gives the details. The resulting
copy of $C_F$ is not asserted to be complemented.

Finite cotype of $E$ is sufficient for \eqref{eq:intro-arazy};
we refer to \cite{Maligranda} for its definition.

\begin{theorem}\label{thm:main}
Let $E,F$ be separable symmetric Banach sequence spaces. If $E$ has
finite cotype and $C_F\hookrightarrow S_E$, then $F=\ell_2$.
\end{theorem}

\begin{corollary}
If $E$ has finite cotype, $1<p<\infty$, and $C_p\hookrightarrow S_E$,
then $p=2$.
\end{corollary}

The proof compares common and independent Rademacher signs on
two-sided disjoint matrix blocks. In the $2$-concave case and in
the $2$-convex, finite-concavity case, this comparison already
follows from Lust-Piquard and Xu \cite[Theorem 1.3]{LPX}; related
product/Bochner comparisons are discussed in \cite{DS,LMS}.
\S~\ref{sec:older-comparisons} records these implications.
Theorem~\ref{thm:main} requires finite cotype, without assuming
$2$-convexity or $2$-concavity. 
Finite cotype is not necessary for
\eqref{eq:intro-arazy} at a fixed $E$, as the case $E=c_0$ already
shows by \cite[Theorem 4.1]{HSS}.

After the preliminaries in Section~\ref{sec:preliminaries},
we prove \eqref{eq:intro-complemented} in Section~\ref{pg:sec-pure}
and Theorem~\ref{pg:uniqueness} in Section~\ref{pg:sec-uniqueness}.
Section~\ref{sec:ordinary-embeddings} treats ordinary embeddings
and compares the hypotheses with those of \cite{HSS}.

\section{Preliminaries}\label{sec:preliminaries}

All sequence spaces are separable symmetric Banach sequence spaces,
normalized by $\norm{e_1}_E=\norm{e_1}_F=1$. We refer to \cite{Arazy,HSS,Simon}
for the background on symmetric sequence spaces and their operator ideals.
We give the definitions and basic properties for $E$; they also apply to $F$.
For a separable Hilbert space $H$, we write $C_E(H)$ for the
corresponding ideal of compact operators on $H$, with norm
$\norm{s(x)}_E$.
We omit $H$ when $H=\ell_2$. Finite matrices are norm dense in $C_E$.
The ideal property gives
\[
 \norm{axb}_{C_E}\leq\norm{a}\norm{x}_{C_E}
                         \norm{b}.
\]
If $U,V$ are Hilbert-space isometries, the map $x\mapsto UxV^*$
preserves singular values.
A rank-one operator has ideal norm equal to its operator norm; hence a
single row or column has its Euclidean norm and
$\abs{x_{ij}}\leq\norm{x}\leq\norm{x}_{C_E}$.
If $x$ has rank at most $d$, then
\[
 \norm{x}_{C_E}\leq\norm{x}_1\leq d\norm{x}.
\]

We use the usual matrix trace. Decomposing
$\ell_2=\bigoplus_{n\geq1}\K^n$, we identify $S_E$ with the closed
subspace of $C_E$ consisting of operators $\diag(x_1,x_2,\ldots)$,
where $x_n\in M_n(\K)$.
For two families $(u_i)$ and $(v_i)$ of orthogonal Hilbert-space projections,
with pairwise orthogonality within each family, the series
$\sum_i u_i x v_i$ converges in ideal norm and satisfies
\begin{equation}\label{eq:orthogonal-projections}
 \norm{\sum_i u_i x v_i}_{C_E}\leq\norm{x}_{C_E}
 \qquad(x\in C_E);
\end{equation}
see \cite[(2.1)]{Arazy81Geometry}; for coordinate projections, see
also \cite[(1.3)]{Arazy78}. The norm convergence follows under
the standing separability assumption on $E$.
Let $q_n$ project onto the summand $\K^n$. Taking $u_n=v_n=q_n$
in \eqref{eq:orthogonal-projections} gives the contractive projection
$x\mapsto\sum_n q_nxq_n$ onto $S_E$. Rank-one coordinate projections
give the projection onto the diagonal matrices.
If Hilbert
projections $p_n\to I$ strongly, then $p_nx\to x$ and $xp_n\to x$ in
ideal norm: first check this for finite-rank $x$, using $\norm{z}_{C_E}\leq\operatorname{rank}(z)\norm{z}$, and then approximate an arbitrary $x$ by finite-rank
operators.

Let $(e_i)_{i\geq1}$ be the canonical orthonormal basis of $\ell_2$.
For $i,j\geq1$, let $e_{ij}$ denote the matrix unit whose only nonzero
entry is a $1$ in row $i$ and column $j$; thus
$e_{ij}e_k=\delta_{jk}e_i$.
Let $p_n$ project onto the first $n$ coordinates of $\ell_2$, with $p_0=0$,
and set $P_nx=p_nxp_n$ and
\begin{equation}\label{pg:shells}
 \Pi_n=P_n-P_{n-1}.
\end{equation}
Thus $\Pi_nC_E$ is spanned by the units $e_{ij}$ with $\max(i,j)=n$.
We write $x^{\mathsf t}$ for transpose in the fixed Hilbert bases.
It is a linear isometry on every $C_E$.
Hilbert inner products are linear in the first variable.
Equality $E=F$ means equality of sequence sets with equivalent norms.
Such equality follows from uniform two-sided estimates on $c_{00}$:
extend the identity by completeness and identify its extensions using
the continuous coordinate functionals.

The shell decomposition and its reproducibility are classical; see
\cite[pp.~86--87]{AL75} and \cite[p.~1]{Arazy83}.
Lemma~\ref{pg:full-reproduction} gives a quantitative version for
embeddings into $C_E$ and the subspace $S_E$, including $F=\ell_1$.
When the original embedding has a bounded left inverse, so does
the resulting embedding.

\begin{lemma}\label{pg:full-reproduction}
Let $Y$ be a Banach space admitting a sequence of finite-rank projections
$H_N:Y\to Y$, $N\geq0$, such that $H_0=0$,
$H_NH_M=H_{\min(N,M)}$ for $N,M\geq0$, $\norm{H_N}\leq1$, and
$H_Ny\to y$ in norm for every $y\in Y$. Suppose that $T:C_F\to Y$ is a linear map
satisfying
\[
 a\norm{x}_{C_F}\leq\norm{Tx}_Y\leq b\norm{x}_{C_F},\qquad a>0.
\]
For every $0<\eta<a$ there are disjoint coordinate isometries $U_0,U_1$,
integers $0=N_0<N_1<\cdots$, and an embedding $V:C_F\to Y$ such that,
with $Jx=U_0xU_0^*-U_1xU_1^*$ and $Q_k=H_{N_k}-H_{N_{k-1}}$,
\[
 \norm{V-TJ}<\eta,\qquad Q_kV=V\Pi_k.
\]
Consequently,
\[
 (a-\eta)\norm{x}_{C_F}\leq\norm{Vx}_Y
                      \leq(2b+\eta)\norm{x}_{C_F}.
\]
If, in addition, a bounded linear map $S:Y\to C_F$ satisfies $ST=I$
and $\eta\norm{S}<1$, then $V$ has a bounded linear left inverse
$R_1:Y\to C_F$ satisfying
\[
 R_1V=I,\qquad
 \norm{R_1}\leq\frac{\norm{S}}{1-\eta\norm{S}}.
\]
In particular, $V(C_F)$ is complemented in $Y$.
\end{lemma}

\begin{proof}
Choose $\epsilon_k>0$ with $\sum_k(2k-1)\epsilon_k<\eta$, and put $N_0=0$.
Suppose that $n_i,m_i$ and $N_i$ have been chosen for $1\leq i<k$,
and put $H=H_{N_{k-1}}T$.

For each fixed $r$, the sequences $(e_{rs})_{s\geq1}$ in row $r$
and $(e_{sr})_{s\geq1}$ in column $r$ are weakly null in $C_F$,
since the corresponding row and column spaces are isometric to $\ell_2$.
Since $H$ is bounded and has finite rank, their images under $H$
converge to zero in norm. We can therefore pick an integer $s_0\geq1$,
larger than all previously chosen coordinates, such that for every
$s\geq s_0$ and every $r\in\{n_i,m_i:i<k\}$,
\[
 \norm{He_{rs}}_Y<\epsilon_k/4,\qquad
 \norm{He_{sr}}_Y<\epsilon_k/4.
\]
The bounded sequence $(He_{ss})_{s\geq s_0}$ lies in the finite-dimensional
range of $H$, so it has a Cauchy subsequence. Pick two indices
$s_0\leq n_k<m_k$ from this subsequence such that
\[
 \norm{H(e_{n_kn_k}-e_{m_km_k})}_Y<\epsilon_k/2.
\]
For $\max(i,j)=k$, put $z_{ij}=e_{n_in_j}-e_{m_im_j}$.
Then $\norm{Hz_{ij}}_Y<\epsilon_k/2$. Indeed, for $i=j=k$ this follows
from the preceding choice of $n_k,m_k$. If $i\ne j$, one index equals
$k$ and the other is less than $k$. Since $n_k,m_k\geq s_0$, the row
and column estimates above give
\[
 \norm{Hz_{ij}}_Y
 \leq\norm{He_{n_in_j}}_Y+\norm{He_{m_im_j}}_Y
 <\epsilon_k/4+\epsilon_k/4=\epsilon_k/2.
\]
Moreover, since $H_Ny\to y$ in norm for every $y\in Y$,  we can pick
$N_k>N_{k-1}$ so that
\[
 \norm{(I-H_{N_k})Tz_{ij}}_Y<\epsilon_k/2
 \qquad(\max(i,j)=k),
\]
and define $Ve_{ij}=Q_kTz_{ij}$. Thus
$\norm{Tz_{ij}-Ve_{ij}}_Y<\epsilon_k$.

Define $U_0,U_1$ by $U_0e_i=e_{n_i}$ and $U_1e_i=e_{m_i}$.
Then $Je_{ij}=z_{ij}$. Define $Lz=U_0^*zU_0$.
We have $\norm{L}\leq1$ and $LJ=I$. Together with the triangle
inequality, this gives
\begin{equation}\label{pg:paired-copy}
 \norm{Jx}_{C_F}\geq\norm{x}_{C_F},\qquad
 \norm{Jx}_{C_F}\leq2\norm{x}_{C_F}.
\end{equation}
For a finite matrix $x$, using $\abs{x_{ij}}\leq\norm{x}_{C_F}$, we have
\[
\begin{aligned}
 \norm{(TJ-V)x}_Y
 &\leq\sum_k\sum_{\max(i,j)=k}
       \abs{x_{ij}}\norm{Tz_{ij}-Ve_{ij}}_Y\\
 &\leq\sum_k\sum_{\max(i,j)=k}\abs{x_{ij}}\epsilon_k\\
 &\leq\Bigl(\sum_k(2k-1)\epsilon_k\Bigr)\norm{x}_{C_F}.
\end{aligned}
\]
Consequently,
\[
\begin{aligned}
 \norm{Vx}_Y
 &\geq a\norm{Jx}_{C_F}-\eta\norm{x}_{C_F}
 \overset{\eqref{pg:paired-copy}}{\geq}(a-\eta)\norm{x}_{C_F},\\
 \norm{Vx}_Y
 &\leq b\norm{Jx}_{C_F}+\eta\norm{x}_{C_F}
 \overset{\eqref{pg:paired-copy}}{\leq}(2b+\eta)\norm{x}_{C_F}.
\end{aligned}
\]
for every finite matrix $x$.
Since finite matrices are dense in $C_F$ and $Y$ is complete,
$V$ extends to a bounded linear map on $C_F$.
Since $H_NH_M=H_{\min(N,M)}$, we have
$Q_kQ_\ell=\delta_{k\ell}Q_\ell$.
For $\max(i,j)=\ell$, the definition $Ve_{ij}=Q_\ell Tz_{ij}$ therefore gives
\[
 Q_kVe_{ij}=\delta_{k\ell}Ve_{ij}=V\Pi_ke_{ij}.
\]
By linearity, $Q_kV=V\Pi_k$ holds on finite matrices.
Since both maps are bounded and finite matrices are dense in $C_F$,
the identity holds on all of $C_F$.

Finally, suppose that $ST=I$ and $\eta\norm{S}<1$.
Set $B=LS$. Then $BTJ=I$ and
\[
 \norm{BV-I}=\norm{B(V-TJ)}<\eta\norm{S}<1.
\]
By \cite[Proposition 2.5]{Douglas}, $R_1=(BV)^{-1}B$, with $R_1V=I$ and the norm bound. Hence $VR_1$ is a bounded projection onto $V(C_F)$.
\end{proof}

\begin{remark}\label{pg:diagonal-endpoint}
The sequence of diagonal matrix units $e_{11},e_{22},e_{33},\ldots$ is not weakly null in $C_{\ell_1}=C_1$. Indeed, the trace is a bounded
linear functional on $C_1$ \cite[Theorem 3.1]{Simon}, and
$\operatorname{Tr}(e_{ss})=1$ for every $s\geq1$.
The proof of Lemma~\ref{pg:full-reproduction} uses a Cauchy subsequence
of their images under the maps $H_NT$ and takes pairwise differences.
It therefore applies without a weak-nullity assumption on the
diagonal units.
\end{remark}

\begin{example}
The following coordinate models illustrate the notation in
Lemma~\ref{pg:full-reproduction}. First, take
\[
 U_0e_i=e_{2i-1},\qquad U_1e_i=e_{2i}\quad(i\geq1).
\]
The operator $U_0xU_0^*$ places a copy of $x$ in the odd rows and
columns, whereas $U_1xU_1^*$ places a copy in the even rows and columns.
Thus $Jx$ consists of two copies on disjoint coordinate subspaces,
with a minus sign on the second copy. Figure~\ref{fig:coordinate-copies}
shows this for $x$ supported on the first two coordinates.
Reordering the first four coordinates as $(1,3,2,4)$ puts these copies
into two diagonal blocks. In general, relative to
$\ell_2=U_0\ell_2\oplus U_1\ell_2$, the operator $Jx$ has the block form
$x\oplus(-x)$.

\begin{figure}[htbp]
\centering
\newcommand{\pgpos}[1]{\textcolor{pgcopypositive}{#1}}
\newcommand{\pgneg}[1]{\textcolor{pgcopynegative}{#1}}
\setlength{\arraycolsep}{4pt}
\[
 \begin{pmatrix}a&b\\c&d\end{pmatrix}
 \xrightarrow{\ J\ }
 \begin{pmatrix}
  \pgpos{a}&0&\pgpos{b}&0\\
  0&\pgneg{-a}&0&\pgneg{-b}\\
  \pgpos{c}&0&\pgpos{d}&0\\
  0&\pgneg{-c}&0&\pgneg{-d}
 \end{pmatrix}
 \xrightarrow{\ (1,3,2,4)\ }
 \left(\begin{array}{cc|cc}
  \pgpos{a}&\pgpos{b}&0&0\\
  \pgpos{c}&\pgpos{d}&0&0\\\hline
  0&0&\pgneg{-a}&\pgneg{-b}\\
  0&0&\pgneg{-c}&\pgneg{-d}
 \end{array}\right).
\]
\caption{The two coordinate copies in $Jx$. The last matrix uses the
same reordering of both rows and columns; all undisplayed entries are zero.}
\label{fig:coordinate-copies}
\end{figure}

For the target projections, take $Y=\ell_2$ and
\[
 H_N(y_1,y_2,\ldots)=(y_1,\ldots,y_N,0,0,\ldots).
\]
Choose $N_0=0$, $N_1=2$, and $N_2=5$. Then
\[
 Q_1y=H_2y,\qquad
 Q_2y=(H_5-H_2)y=(0,0,y_3,y_4,y_5,0,\ldots).
\]
Thus $H_{N_k}$ retains the first $k$ selected blocks, while $Q_k$
retains only the $k$th block, as shown in
Figure~\ref{fig:coordinate-projections}. More generally,
$Q_ky=\sum_{j=N_{k-1}+1}^{N_k}y_je_j$ in this coordinate model.

\begin{figure}[htbp]
\centering
\newcommand{\pgold}[1]{\colorbox{black!8}{\strut$#1$}}
\newcommand{\pgnew}[1]{\colorbox{pgcopypositive!15}{\strut$#1$}}
\setlength{\tabcolsep}{5pt}
\renewcommand{\arraystretch}{1.35}
\begin{tabular}{c|cccccccc|c}
 Coordinate & $1$ & $2$ & $3$ & $4$ & $5$ & $6$ & $7$ & $8$ & $\cdots$\\\hline
 $y$ & $y_1$ & $y_2$ & $y_3$ & $y_4$ & $y_5$ & $y_6$ & $y_7$ & $y_8$ & $\cdots$\\
 $H_2y$ & \pgold{y_1} & \pgold{y_2} & $0$ & $0$ & $0$ & $0$ & $0$ & $0$ & $\cdots$\\
 $H_5y$ & \pgold{y_1} & \pgold{y_2} & \pgnew{y_3} & \pgnew{y_4} & \pgnew{y_5} & $0$ & $0$ & $0$ & $\cdots$\\\hline
 $Q_2y$ & $0$ & $0$ & \pgnew{y_3} & \pgnew{y_4} & \pgnew{y_5} & $0$ & $0$ & $0$ & $\cdots$
\end{tabular}
\caption{The difference $Q_2=H_5-H_2$ retains precisely coordinates
$3,4,5$. The shaded entries of $H_2y$ cancel when it is subtracted
from $H_5y$.}
\label{fig:coordinate-projections}
\end{figure}

These models illustrate the maps separately. In
Lemma~\ref{pg:full-reproduction}, the coordinate copies and the cutoffs
$N_k$ are chosen using $T$. The identity $Q_kV=V\Pi_k$ means that
applying $V$ to the $k$th source shell gives exactly the $k$th target
block of $Vx$:
\[
 V(\Pi_kx)=Q_k(Vx)\qquad(x\in C_F).
\]
\end{example}

\section{Complemented embeddings into the subspace \texorpdfstring{$S_E$}{S\_E}}
\label{pg:sec-pure}

Let $(d_t)_{t\geq1}$ be positive integers and let
\[
 Z=\overline{\bigoplus_{t\geq1}M_{d_t}(\K)}^{\,C_E},\qquad
 \norm{(z_t)_t}_Z=\norm{s(\diag(z_1,z_2,\ldots))}_E.
\]
For $d_t=t$, this space is $S_E$. Finite direct sums of the spaces
$M_{d_t}(\K)$ carry the norm inherited from $Z$.
For $N\ge0$, define the coordinate projection $H_N:Z\to Z$ by
\[
 (H_Nz)_t=
 \begin{cases}
 z_t,&1\le t\le N,\\
 0,&t>N.
 \end{cases}
\]
In particular, $H_0=0$.
We use the projections $\Pi_n$ defined in \eqref{pg:shells}.

\begin{proposition}\label{pg:reproduction}
Suppose that \(T:C_F\to Z\) satisfies
\[
 a\norm{x}_{C_F}\le \norm{Tx}_Z\le b\norm{x}_{C_F}
 \qquad (x\in C_F),
 \qquad a>0.
\]
For every \(0<\eta<a\), there are increasing disjoint coordinate
sequences \((n_i)\), \((m_i)\), integers
\(0=N_0<N_1<N_2<\cdots\), and an embedding \(V:C_F\to Z\)
such that, with
\[
 \begin{gathered}
 U_ne_i=e_{n_i},\qquad U_me_i=e_{m_i},\\
 Jx=U_nxU_n^*-U_mxU_m^*,\qquad Q_k=H_{N_k}-H_{N_{k-1}},
 \end{gathered}
\]
one has
\begin{equation}\label{pg:reproduction-bounds}
 \begin{gathered}
 \norm{V-TJ}<\eta,\qquad Q_kV=V\Pi_k,\\
 (a-\eta)\norm{x}_{C_F}\le\norm{Vx}_Z
 \le(2b+\eta)\norm{x}_{C_F}.
 \end{gathered}
\end{equation}
If in addition \(S:Z\to C_F\) is bounded and \(ST=I\), choose
\(\eta\norm{S}<1\). There is then a bounded map \(R:Z\to C_F\) such that
\begin{equation}\label{pg:reproduction-return}
 \begin{gathered}
 RV=I,\qquad \Pi_kR=RQ_k\quad(k\ge1),\\
 \norm{R}\le
 \frac{\norm{V}}{a-\eta}\,
 \frac{\norm{S}}{1-\eta\norm{S}}.
 \end{gathered}
\end{equation}
\end{proposition}

\begin{proof}
Apply Lemma~\ref{pg:full-reproduction} with $Y=Z$ and the projections $H_N$.
It gives $V,J,Q_k$ satisfying \eqref{pg:reproduction-bounds}.
Suppose now that $ST=I$ and $\eta\norm{S}<1$.
Lemma~\ref{pg:full-reproduction} gives a left inverse $R_1$ with
\[
 R_1V=I,\qquad \norm{R_1}\leq\frac{\norm{S}}{1-\eta\norm{S}}.
\]
We average over the choices of signs $\tau_k\in\{-1,1\}$.

Let \(\tau_k\in\{-1,1\}\), with only finitely many signs different
from \(1\), and define
\[
 M_\tau=I+\sum_k(\tau_k-1)\Pi_k,\qquad
 D_\tau=I+\sum_k(\tau_k-1)Q_k.
\]
For \(z=(z_t)_t\in Z\),
\[
 (D_\tau z)_t=\tau_kz_t
 \qquad(N_{k-1}<t\le N_k).
\]
Multiplication of \(z_t\in M_{d_t}(\K)\) by \(\tau_k\in\{-1,1\}\)
preserves its singular values, so \(D_\tau\) is an isometry on \(Z\).
Since \(VM_\tau=D_\tau V\), the lower bound for \(V\) yields
\begin{equation}\label{pg:shell-sign-bound}
 \norm{M_\tau}\le K_0,\qquad K_0=\frac{\norm{V}}{a-\eta}.
\end{equation}

For \(y=\sum_{k=1}^N y_k\), where \(y_k=Q_ky_k\), define
\[
 Ry=\sum_{k=1}^N\Pi_kR_1y_k.
\]
For this $y$, one has
\begin{equation}\label{pg:finite-return-average}
 Ry=\E_\tau[M_\tau R_1D_\tau y],
\end{equation}
where \(\tau_1,\ldots,\tau_N\) are independent signs and all later
signs are \(1\). To check the identity, expand
\(M_\tau=I+\sum_{\ell=1}^N(\tau_\ell-1)\Pi_\ell\).
The coefficient of \(R_1y_k\) in the identity term has mean zero,
whereas
\[
 \E_\tau[(\tau_\ell-1)\tau_k]=\delta_{\ell k}.
\]
Thus the average equals \(\sum_{k=1}^N\Pi_kR_1y_k\) and belongs
to \(P_NC_F\). Equations
\eqref{pg:shell-sign-bound}--\eqref{pg:finite-return-average}
give \(\norm{Ry}\le K_0\norm{R_1}\norm{y}\), uniformly in \(N\).
The formulas for \(R\) agree on the nested subspaces
\(H_{N_m}Z\), \(m\ge1\). Their union is dense in \(Z\), since
\(N_m\to\infty\), so \(R\) extends boundedly to \(Z\).

For a finite matrix \(x\),
\[
 RVx
 =\sum_k\Pi_kR_1V\Pi_kx
 =\sum_k\Pi_kx=x.
\]
The definition also gives \(\Pi_kR=RQ_k\) on
\(\bigcup_{m\ge1}H_{N_m}Z\).
Both identities extend by continuity. Since
$\norm{R_1}\le\norm{S}/(1-\eta\norm{S})$, this proves
\eqref{pg:reproduction-return}.
\end{proof}

For a finite set \(G\subset\N\), put
\[
 Z_G=\{z=(z_t)_t\in Z:z_t=0\text{ for }t\notin G\}
     =\bigoplus_{t\in G}M_{d_t}(\K),
\]
with the norm inherited from \(Z\). A pair of block multipliers
on \(Z_G\) means families \(a=(a_t)_{t\in G}\),
\(b=(b_t)_{t\in G}\), with \(a_t,b_t\in M_{d_t}(\K)\), satisfying
\[
 \norm{a_t}\le1,\qquad \norm{b_t}\le1\quad(t\in G).
\]
The map \(z\mapsto azb\), defined by \((azb)_t=a_tz_tb_t\), is
contractive by the ideal property.

\begin{lemma}\label{pg:mixed-inequality}
Let \(G\subset\N\) be finite, and let
\(B_1,B_2,D_1,D_2\in Z_G\), with matrix coordinates
\(B_{j,t},D_{i,t}\in M_{d_t}(\K)\), \(t\in G\). Define
\[
 \begin{gathered}
 u_{ij}(a,b)=\sum_{t\in G}\operatorname{Tr}
                  (D_{i,t}^*a_tB_{j,t}b_t),\\
 c_{ij}=\max_{a,b}\abs{u_{ij}(a,b)},\qquad
 \Delta=\max_{a,b}\abs{u_{11}(a,b)-u_{22}(a,b)},
 \end{gathered}
\]
where the maxima range over the block multipliers just defined.
Then
\begin{equation}\label{pg:mixed-estimate}
 c_{11}\le\sqrt{c_{21}c_{12}}+\Delta.
\end{equation}
Moreover,
\begin{equation}\label{pg:block-maximum}
 c_{ij}=\sum_{t\in G}\sum_\ell
                 s_\ell(B_{j,t})s_\ell(D_{i,t}).
\end{equation}
\end{lemma}

\begin{proof}
First fix \(d\ge1\) and take \(B,D\in M_d(\K)\). Then
\begin{equation}\label{pg:trace-maximum}
 \max_{\norm{a},\norm{b}\le1}
 \abs{\operatorname{Tr}(D^*aBb)}
 =\sum_\ell s_\ell(B)s_\ell(D).
\end{equation}
For the upper bound, use von Neumann's trace inequality
\cite{Bhatia} and $s_\ell(aBb)\leq s_\ell(B)$ for contractions $a,b$.
Singular value decompositions show that the bound is attained.

To prove \eqref{pg:mixed-estimate} when $G=\{t\}$, regard
$B_1,B_2,D_1,D_2$ as matrices in $M_{d_t}(\K)$. Choose decompositions
\[
 B_1=UXV^*,\qquad D_1^*=WYZ^*,
\]
where $X=\diag(x_i)$ and $Y=\diag(y_i)$ have nonnegative decreasing
entries. In all four forms make the same change of variables
\[
 a'=Z^*aU,\qquad b'=V^*bW,\qquad
 B'_j=U^*B_jV,\qquad C_i=W^*D_i^*Z.
\]
Then $B'_1=X$, $C_1=Y$, and
\[
 \operatorname{Tr}(D_i^*aB_jb)
       =\operatorname{Tr}(C_i a'B'_j b').
\]
The change of variables maps the contraction balls onto themselves,
so all four maxima and $\Delta$ are preserved.

Put \(z_i=(C_2)_{ii}(B'_2)_{ii}\). For independent signs
\(\varepsilon_i\), take
\[
 a'=\diag(\varepsilon_i),\qquad
 b'=\diag(\varepsilon_i t_i),\qquad t_i\in\K,\quad \abs{t_i}=1.
\]
Averaging $u_{11}-u_{22}$ eliminates the off-diagonal terms and
gives \(\sum_i(x_iy_i-z_i)t_i\).
Its absolute value is at most \(\Delta\). Choose \(t_i\) to align
these terms and obtain
\begin{equation}\label{pg:diagonal-discrepancy}
 \sum_i\abs{x_iy_i-z_i}\le\Delta.
\end{equation}
Diagonal choices of $a'$ and $b'$ in $u_{21}$ and $u_{12}$ give
\[
 \sum_i x_i\abs{(C_2)_{ii}}\le c_{21},\qquad
 \sum_i y_i\abs{(B'_2)_{ii}}\le c_{12}.
\]
For \(w\ge0\) and \(z\in\K\),
\(w\le\sqrt{w\abs{z}}+\abs{w-z}\). Indeed, if \(\abs{z}\ge w\), the
square-root term suffices; if \(\abs{z}\le w\), use
\(\abs{w-z}\ge w-\abs{z}\) and \(\sqrt{w\abs{z}}\ge\abs{z}\).
Consequently, \eqref{pg:trace-maximum},
\eqref{pg:diagonal-discrepancy}, and Cauchy--Schwarz yield
\[
 c_{11}=\sum_i x_iy_i
 \le
 \left(\sum_i x_i\abs{(C_2)_{ii}}\right)^{1/2}
 \left(\sum_i y_i\abs{(B'_2)_{ii}}\right)^{1/2}
 +\Delta
 \le\sqrt{c_{21}c_{12}}+\Delta.
\]

For a general finite set $G$ and each \(t\in G\), let
\(c_{ij}^{\,t}\) and \(\Delta_t\) be the
maxima defining \(c_{ij}\) and \(\Delta\), respectively, with the
trace sums restricted to the single index \(t\). All maxima are attained, since the
contraction balls are compact. For each \(t\in G\), independently
multiply a maximizing left factor \(a_t\) by a scalar of modulus one
to make its value nonnegative.
The triangle inequality and these aligned choices show that
\[
 c_{ij}=\sum_{t\in G}c_{ij}^{\,t},\qquad
 \Delta=\sum_{t\in G}\Delta_t.
\]
The second identity holds because $u_{11}-u_{22}$ is homogeneous
in the left factor. Summing the
single-block inequality and applying Cauchy--Schwarz gives
\[
 c_{11}
 \le\sum_{t\in G}\sqrt{c_{21}^{\,t}c_{12}^{\,t}}
      +\sum_{t\in G}\Delta_t
 \le\sqrt{c_{21}c_{12}}+\Delta.
\]
Applying \eqref{pg:trace-maximum} to
\(B_{j,t},D_{i,t}\in M_{d_t}(\K)\) for each \(t\in G\) proves
\eqref{pg:block-maximum}. For each fixed \(t\), the singular values
of \(B_{j,t}\) and \(D_{i,t}\) are listed in decreasing order.
\end{proof}

For \(d\ge2\), an invariant orthonormal pair in \(\K^d\) means two
specified distinct columns of a Haar-distributed element of
\(\{U\in M_d(\K):U^*U=I_d\}\).

\begin{lemma}\label{pg:haar-return}
Let \(G\subset\N\) be finite, and let \(Z_G\) have
the norm inherited from \(Z\). Suppose that
\[
 J:\ell_2^d\longrightarrow Z_G,\qquad
 T:Z_G\longrightarrow\ell_2^d,\qquad TJ=I,\qquad d\ge2.
\]
For unit vectors \(u,w\), set
\[
 \beta(u,w)(a,b)=\langle T(aJ(u)b),w\rangle,\qquad
 c(u,w)=\max_{a,b}\abs{\beta(u,w)(a,b)}.
\]
Then, for an invariant orthonormal pair \((u,w)\),
\begin{equation}\label{pg:haar-lower-bound}
 \E c(u,w)\ge\frac13.
\end{equation}
The function \(c\) is continuous, and, with \(k=\norm{J}\norm{T}\),
\begin{equation}\label{pg:coefficient-continuity}
 \begin{gathered}
 0\le c(u,w)\le k,\\
 \abs{c(u,w)-c(u',w')}
 \le k\bigl(\norm{u-u'}_2+\norm{w-w'}_2\bigr)
 \end{gathered}
\end{equation}
for unit vectors \(u,w,u',w'\).
\end{lemma}

\begin{proof}
Contractivity of the block multipliers proves
\eqref{pg:coefficient-continuity}, first for each scalar form and
then for its maximum. For an orthonormal pair \(u,w\), take
\(B_1=Ju\), \(B_2=Jw\). Finite-dimensional trace duality represents
the functionals \(z\mapsto\langle Tz,u\rangle\) and
\(z\mapsto\langle Tz,w\rangle\) by matrices \(D_1,D_2\in Z_G\),
with \(D_{i,t}\in M_{d_t}(\K)\), \(t\in G\). Since \(TJ=I\),
\[
 \sum_{t\in G}\operatorname{Tr}(D_{i,t}^*B_{j,t})
 =\delta_{ij}.
\]
In particular \(c(u,u)\ge1\), by choosing identity multipliers.
Lemma~\ref{pg:mixed-inequality} therefore gives
\begin{equation}\label{pg:haar-pointwise}
 1\le
 \sqrt{c(u,w)c(w,u)}
 +\norm{\beta(u,u)-\beta(w,w)},
\end{equation}
where
\[
 \norm{\beta(u,u)-\beta(w,w)}
 =\max_{a,b}\abs{\beta(u,u)(a,b)-\beta(w,w)(a,b)},
\]
and \(a,b\) range over the block multipliers on \(Z_G\).

Put \(u_+=(u+w)/\sqrt2\), \(u_-=(u-w)/\sqrt2\). They are
orthonormal, and expansion gives the identity of scalar forms
\[
 \beta(u,u)-\beta(w,w)
 =\beta(u_+,u_-)+\beta(u_-,u_+).
\]
The last term in \eqref{pg:haar-pointwise} is thus
at most \(c(u_+,u_-)+c(u_-,u_+)\).

The distribution of an invariant orthonormal pair is unchanged
by swapping its entries or by this fixed two-dimensional rotation.
If \(A=\E c(u,w)\), the arithmetic-geometric mean inequality and
the swap invariance give
\[
 \E\sqrt{c(u,w)c(w,u)}
 \le\frac12\E\bigl(c(u,w)+c(w,u)\bigr)=A.
\]
Each rotated cross term also has expectation \(A\).
Averaging \eqref{pg:haar-pointwise} proves \(1\le3A\).
\end{proof}

\begin{theorem}\label{pg:pure-hilbert}
Let \(\N=\bigsqcup_{n\ge1}G_n\), where each \(G_n\) is finite. Define \(Q_n:Z\to Z\) by
\[
 (Q_nz)_t=
 \begin{cases}
 z_t,&t\in G_n,\\
 0,&t\notin G_n.
 \end{cases}
\]
Thus \(Q_nZ=Z_{G_n}\). Suppose that bounded maps
\[
 V:C_F\longrightarrow Z,\qquad R:Z\longrightarrow C_F
\]
satisfy
\begin{equation}\label{pg:pure-identities}
 RV=I,\qquad Q_nV=V\Pi_n,\qquad \Pi_nR=RQ_n\quad(n\ge1).
\end{equation}
Writing \(\kappa=\norm{V}\norm{R}\), one has
\begin{equation}\label{pg:pure-equivalence}
 \frac{1}{3\kappa}\norm{\alpha}_2
 \le\norm{\alpha}_F
 \le3\kappa\norm{\alpha}_2
 \qquad(\alpha\in c_{00}).
\end{equation}
In particular \(F=\ell_2\) canonically, with equivalent norms.
\end{theorem}

\begin{proof}
Fix \(d\ge2\). For \(n>d\) and \(u\in\K^d\), extended by zero
outside the first \(d\) coordinates, put
\[
 \begin{gathered}
 \rho_n(u)=\sum_{j=1}^d u_je_{nj}=e_nu^\top,\\
 J_nu=V\rho_n(u),\qquad
 T_nz=((Rz)_{nj})_{j=1}^d\quad(z\in Q_nZ).
 \end{gathered}
\]
The map \(\rho_n:\ell_2^d\to C_F\) is linear and isometric.
The row compression \(x\mapsto e_{nn}xp_d\) is contractive and
has Euclidean coefficient norm. Equations
\eqref{pg:pure-identities} imply
\[
 J_n(\ell_2^d)\subset Q_nZ,\qquad
 T_nJ_n=I,\qquad
 \norm{J_n}\le\norm{V},\qquad \norm{T_n}\le\norm{R}.
\]
For unit vectors \(u,w\), define
\[
 c_n(u,w)=
 \max_{a,b}
 \abs{\langle T_n(aJ_n(u)b),w\rangle},
\]
where \(a_t,b_t\in M_{d_t}(\K)\), \(t\in G_n\), range over
the block multipliers on \(Z_{G_n}\).
Lemma~\ref{pg:haar-return} gives
\begin{equation}\label{pg:row-mean}
 \E c_n(u,w)\ge\frac13
\end{equation}
for an invariant orthonormal pair. In addition \(c_n\) is continuous
and at most \(\kappa\), with the uniform continuity estimate
\eqref{pg:coefficient-continuity}.

Let \(\alpha=(\alpha_1,\ldots,\alpha_m)\), put \(d=m+1\), and
choose distinct \(n_1,\ldots,n_m>d\). For one Haar-distributed
orthonormal basis \(u_0,u_1,\ldots,u_m\) of \(\K^d\), set
\[
 c_s=c_{n_s}(u_s,u_0)\quad(1\le s\le m).
\]
For each \(s\), choose a maximizing pair of block multipliers on
\(Z_{G_{n_s}}\). Multiply the left factors \(a_t\), \(t\in G_{n_s}\),
by a common scalar of modulus one so that
\(\langle T_{n_s}(aJ_{n_s}(u_s)b),u_0\rangle=c_s\ge0\).
Since the sets \(G_{n_1},\ldots,G_{n_m}\) are pairwise disjoint,
the chosen multipliers define a map
\(\mathcal A:Z\to Z\) of the form
\[
 (\mathcal Az)_t=a_tz_tb_t,\qquad
 \sup_t\norm{a_t},\ \sup_t\norm{b_t}\le1.
\]
Set \(a_t=b_t=I_{d_t}\) for
\(t\notin\bigcup_{s=1}^mG_{n_s}\). Thus \(\norm{\mathcal A}\le1\),
and \(\mathcal A(Z_{G_n})\subset Z_{G_n}\) for every \(n\).
The choice of \(\mathcal A\) may depend on the basis.

First consider
\[
 x=\sum_{s=1}^m\alpha_s\rho_{n_s}(u_s).
\]
Its row directions \(u_s\) are orthonormal and its row coordinates
\(n_s\) are distinct. Its nonzero singular values are therefore the
\(\abs{\alpha_s}\), and \(\norm{x}_{C_F}=\norm{\alpha}_F\).
By \eqref{pg:pure-identities},
\[
 Vx\in\bigoplus_{s=1}^mZ_{G_{n_s}},\qquad
 R\mathcal AVx\in\sum_{s=1}^m\Pi_{n_s}C_F.
\]

Set \(q_0=\overline{u_0}u_0^\top\), regarded as an operator on the
source Hilbert space by extension by zero. This is the orthogonal
projection onto the direction \(\overline{u_0}\). Direct multiplication gives
\[
 \rho_n(v)q_0=\langle v,u_0\rangle\rho_n(u_0).
\]
Since \(n_s>d\), multiplication on the right by \(q_0\)
annihilates the column-\(n_s\) part of the source subspace
\(\Pi_{n_s}C_F\), including its diagonal entry. It retains only
the row-\(n_s\) part projected onto the first \(d\) coordinates.
The identities
\(\langle T_{n_s}(aJ_{n_s}(u_s)b),u_0\rangle=c_s\)
give
\begin{equation}\label{pg:common-output}
 (R\mathcal AVx)q_0
 =\sum_{s=1}^m\alpha_sc_s\rho_{n_s}(u_0).
\end{equation}
The right-hand side has rank at most one and norm
\(\norm{(\alpha_sc_s)_s}_2\). Hence, for every chosen orthonormal basis,
\begin{equation}\label{pg:first-pointwise}
 \norm{(\alpha_sc_s)_s}_2\le\kappa\norm{\alpha}_F.
\end{equation}
Each pair \((u_s,u_0)\) has the invariant orthonormal-pair
distribution, even though its associated row map depends on \(s\).
Thus \eqref{pg:row-mean}, solidity, and Jensen's inequality imply
\begin{equation}\label{pg:first-comparison}
 \frac13\norm{\alpha}_2
 \le\norm{(\alpha_s\E c_s)_s}_2
 \le\E\norm{(\alpha_sc_s)_s}_2
 \le\kappa\norm{\alpha}_F.
\end{equation}

For the reverse comparison put \(c'_s=c_{n_s}(u_0,u_s)\).
Choose multipliers on each \(Z_{G_{n_s}}\), with the left factors
multiplied by a common scalar of modulus one, so that
\(\langle T_{n_s}(aJ_{n_s}(u_0)b),u_s\rangle=c'_s\ge0\).
Set \(a_t=b_t=I_{d_t}\) outside \(\bigcup_{s=1}^mG_{n_s}\), and let
\(\mathcal A z=(a_tz_tb_t)_t\). This is a contraction on \(Z\). Now let
\[
 x=\sum_{s=1}^m\alpha_s\rho_{n_s}(u_0),\qquad y=R\mathcal AVx.
\]
This input has rank at most one and norm \(\norm{\alpha}_2\), so
\(\norm{y}_{C_F}\le\kappa\norm{\alpha}_2\).
Define Hilbert-space isometries \(L,W:\K^m\to\ell_2\) by
\[
 Le_s=e_{n_s},\qquad We_s=\overline{u_s},
\]
and define
\[
 \mathcal D(z)=L\,\diag(L^*zW)\,W^*.
\]
Here \(\diag(A)\) for a finite matrix denotes its diagonal part.
The ideal property and the diagonal case of
\eqref{eq:orthogonal-projections} give \(\norm{\mathcal D}\le1\).

For \(j\le d\), an entry at \((n_s,j)\) belongs only to the
source subspace \(\Pi_{n_s}C_F\). The identity
\(\Pi_{n_s}R=RQ_{n_s}\) in \eqref{pg:pure-identities} therefore
excludes contributions from \(Z_{G_{n_r}}\), \(r\ne s\), to the
\(s\)-th diagonal entry of \(L^*yW\).
That entry is
\[
 e_{n_s}^*y\,\overline{u_s}=\alpha_sc'_s.
\]
Hence
\begin{equation}\label{pg:orthogonal-output}
 \mathcal D(y)=
 \sum_{s=1}^m\alpha_sc'_s\rho_{n_s}(u_s).
\end{equation}
The singular values on the right are the \(\abs{\alpha_sc'_s}\).
Therefore
\[
 \norm{(\alpha_sc'_s)_s}_F\le\kappa\norm{\alpha}_2.
\]
Since \(\E c'_s\ge1/3\), Jensen's inequality in the restriction of
the \(F\)-norm to the first \(m\) sequence coordinates yields
\begin{equation}\label{pg:second-comparison}
 \frac13\norm{\alpha}_F
 \le\norm{(\alpha_s\E c'_s)_s}_F
 \le\E\norm{(\alpha_sc'_s)_s}_F
 \le\kappa\norm{\alpha}_2.
\end{equation}

The integrands depend only on the continuous functions $c_n$;
no measurable selection of maximizing multipliers is needed.
Both comparisons use the same maps $V,R$, and $\kappa$
is independent of $m$.

Combining \eqref{pg:first-comparison} and
\eqref{pg:second-comparison} proves \eqref{pg:pure-equivalence}.
The case \(m=1\) is included by taking \(d=2\).
Finally \(F\) and \(\ell_2\) are the completions of \(c_{00}\)
in their respective norms. The identity on \(c_{00}\) extends
boundedly in both directions. Coordinate evaluations are
continuous and agree on \(c_{00}\), so these extensions identify
the sequence sets. Hence $F=\ell_2$.
\end{proof}

\begin{theorem}\label{pg:comp-block}
Let \(E,F\) be separable symmetric Banach sequence spaces.
If \(C_F\) is isomorphic to a
complemented subspace of \(S_E\), then \(F=\ell_2\) canonically,
with equivalent norms. More generally, the same conclusion holds
with \(S_E\) replaced by
\(Z=\overline{\bigoplus_{t\ge1}M_{d_t}(\K)}^{\,C_E}\)
for any sequence of positive integers \((d_t)_{t\ge1}\).
\end{theorem}

\begin{proof}
Choose an embedding \(T:C_F\to Z\) and a bounded left inverse
\(S:Z\to C_F\), so that \(ST=I\). Such a left inverse is obtained
by composing a bounded projection onto \(T(C_F)\) with the inverse
of \(T\) on its range. Renormalize the sequence norms if necessary
so that their first unit vectors have norm one.
Choose \(a>0\) such that
\(\norm{Tx}_Z\ge a\norm{x}_{C_F}\) for every \(x\in C_F\).
Apply Proposition~\ref{pg:reproduction} with
\(0<\eta<a\) and \(\eta\norm{S}<1\). We obtain bounded maps
\(V,R\) satisfying \eqref{pg:pure-identities},
where the $k$th group consists of $M_{d_t}(\K)$ with
$N_{k-1}<t\leq N_k$.
The conclusion follows from Theorem~\ref{pg:pure-hilbert}.
\end{proof}

\begin{remark}\label{pg:comp-block-scope}
The proof of Theorem~\ref{pg:comp-block} uses the bounded left
inverse $R$. The embeddings in Subsection~\ref{sec:ordinary-limits}
need not have a bounded left inverse.
\end{remark}

\section{Uniqueness of the symmetric structure}\label{pg:sec-uniqueness}
\label{pg:sec-components}\label{pg:sec-returns}

Retain the projections $P_n$ and $\Pi_n$ from
\eqref{pg:shells}. Let $\mathcal D_Ex=\sum_i x_{ii}e_{ii}$ for $x\in C_E$,
and let $L_E$ be strict lower triangular truncation on finite matrices,
with $t_E=\norm{L_E}\in[1,\infty]$. Define $\mathcal D_F$, $L_F$ and $t_F$
in the same way. For a separable infinite-dimensional
Hilbert space $H_0$, we consider maps $A:C_F\to C_E(H_0\otimes\ell_2)$
of the form $A(e_{ij})=b_{ij}\otimes e_{ij}$, where $b_{ij}$ acts on $H_0$.

\begin{lemma}\label{pg:triangle-invariance}
An ordinary embedding $C_F\to C_E$ transfers boundedness of strict lower
triangular truncation from $C_E$ to $C_F$. Hence finiteness of $t_E$ is
invariant under isomorphisms between symmetric compact ideals.
\end{lemma}

\begin{proof}
The upper triangular projection in \cite[Lemma 4.5]{Arazy78} is $I-L_E$.
If $t_E<\infty$, the implication (i)$\Rightarrow$(iii) of that lemma
gives an unconditional finite-dimensional decomposition of $C_E$.
The embedding $C_F\to C_E$ therefore satisfies condition (ii) for $C_F$,
and (ii)$\Rightarrow$(i) gives boundedness of $I-L_F$, hence $t_F<\infty$.

Finally, apply the implication to an isomorphism and its inverse.
\end{proof}

\begin{lemma}\label{pg:coarse-components}
Suppose $V:C_F\to C_E$ and $R:C_E\to C_F$ satisfy $RV=I$ and
$(P_{N_k}-P_{N_{k-1}})V=V\Pi_k$, as in Lemma~\ref{pg:full-reproduction}
with $Y=C_E$ and $H_N=P_N$. Write the
Hilbert space $\ell_2$ on which $C_E$ acts as $\bigoplus_i\mathcal H_i$,
where $\mathcal H_i=\operatorname{span}\{e_n:N_{i-1}<n\leq N_i\}$, and let
$q_i$ project onto $\mathcal H_i$. Thus
\begin{equation}\label{pg:max-support}
 q_pV(e_{ij})q_q=0\quad\text{unless}\quad
 \max(p,q)=\max(i,j).
\end{equation}
There are bounded maps $C_{ab}:C_F\to C_E$, $a,b\in\{1,2,0\}$,
vanishing on $e_{ii}\in C_F$, such that
\[
 \begin{gathered}
 \sum_{a,b}C_{ab}=V(I-\mathcal D_F),\qquad
 \norm{C_{ab}}\leq 2c_ac_b\norm{V}\leq18\norm{V},\\
 c_1=c_2=1,\qquad c_0=3.
 \end{gathered}
\]
For $i>j$ only $C_{11},C_{12},C_{21},C_{10},C_{01}$ can be nonzero, and
\begin{equation}\label{pg:return-sum}
 1=\sum_{ab\in\{11,12,21,10,01\}}
       (RC_{ab}(e_{ij}))_{ij}.
\end{equation}
Here $C_{12}(e_{ij})\in q_iC_Eq_j$ and
$C_{21}(e_{ij})\in q_jC_Eq_i$ for every $i\ne j$.
The transpose $C_{21}(e_{ij})^{\mathsf t}$ also lies in $q_iC_Eq_j$.
Thus the ranges of $C_{12}(e_{ij})$ and $C_{21}(e_{ij})^{\mathsf t}$
lie in $\mathcal H_i$, and the ranges of their adjoints lie in
$\mathcal H_j$. On $i>j$,
$C_{11}(e_{ij})\in q_iC_Eq_i$,
$C_{10}(e_{ij})\in q_iC_E(I-q_i-q_j)$, and
$C_{01}(e_{ij})\in(I-q_i-q_j)C_Eq_i$. The latter three components vanish on
all $e_{ij}\in C_F$ with $i\leq j$. For $i>j$,
$q_iC_{10}(e_{ij})q_q=0$ when $q\geq i$, and
$q_pC_{01}(e_{ij})q_i=0$ when $p\geq i$.

For one of these five components, $\abs{(RC_{ab}(e_{ij}))_{ij}}\geq1/5$
on all pairs $i>j$ with $i,j$ in one infinite subset of $\mathbb N$.
\end{lemma}

\begin{proof}
On the space of bounded maps $A:C_F\to C_E$, define
\[
 \begin{array}{ll}
 \mathcal P_{LL}(A)e_{ij}=q_iA(e_{ij}),&
 \mathcal P_{LR}(A)e_{ij}=q_jA(e_{ij}),\\
 \mathcal P_{RL}(A)e_{ij}=A(e_{ij})q_i,&
 \mathcal P_{RR}(A)e_{ij}=A(e_{ij})q_j.
 \end{array}
\]
Each has operator norm at most one. For example, if the source row support
of $x$ is in $\{1,\ldots,m\}$, average
$D_\epsilon A(d_\epsilon x)$ over independent signs on these indices,
where $D_\epsilon$ acts as $\epsilon_i I_{\mathcal H_i}$ on
$\mathcal H_i$ for $i\leq m$, $d_\epsilon e_i=\epsilon_i e_i$ for
$i\leq m$, and both operators are the identity on the respective
orthogonal complements. For $e_{ij}$ the average is $q_iA(e_{ij})$;
all terms $q_kA(e_{ij})$ with $k\ne i$, including $k>m$, cancel.
Density gives the first map and its bound. The other three follow by
using source or target right multiplication in the same average.

These four operations are commuting projections, as is checked on matrix
units by the commuting left and right projections $q_i,q_j$. On maps
vanishing on diagonal source units, put
\[
 \mathcal L_1=\mathcal P_{LL},\quad
 \mathcal L_2=\mathcal P_{LR},\quad
 \mathcal L_0=I-\mathcal P_{LL}-\mathcal P_{LR},
\]
and define $\mathcal M_1,\mathcal M_2,\mathcal M_0$ by the analogous
target-right operations. Each triple consists of disjoint projections;
its norms are bounded by $1,1,3$. Define
$C_{ab}=\mathcal L_a\mathcal M_b(V(I-\mathcal D_F))$. This proves the
sum identity and the norm bounds. For $i\ne j$, the row index retained
by $a=1,2,0$ is
respectively $i,j$, or neither, and the analogous statement holds for the
column index and $b$. The maximum-support condition proves all the stated
vanishings and support descriptions, including $C_{00}=0$.

Since $RV=I$, \eqref{pg:return-sum} follows. At least one of its five
summands has modulus at least $1/5$. Color each unordered pair by one such
component on its lower ordered pair. The elementary infinite Ramsey
theorem for pairs gives a homogeneous infinite set: choose successively
a vertex and an infinite set of remaining vertices on which
its incident-edge color is constant. One of the recorded colors occurs
infinitely often; the vertices with that color form the required set.
\end{proof}

\begin{lemma}\label{pg:matched-return}
Let $A:C_F\to C_E(H)$ and $R:C_E(H)\to C_F$ be bounded. Suppose
$(u_i)$ and $(v_j)$ are families of orthogonal projections on $H$, with
mutual orthogonality within each family, and
\[
 A(e_{ij})=u_iA(e_{ij})v_j.
\]
There is a bounded map $B:C_E(H)\to C_F$ with
$\norm{B}\leq\norm{R}$, such that
\[
 BA=M_\lambda,\qquad
 \lambda_{ij}=(RA(e_{ij}))_{ij},
 \qquad M_\lambda(e_{ij})=\lambda_{ij}e_{ij}.
\]
On $uC_E(H)v$, where $u=\sum_i u_i$ and $v=\sum_jv_j$
are strong sums, $B$ satisfies
\[
 B(u_i y)=e_{ii}B(y),\qquad B(yv_j)=B(y)e_{jj}.
\]
In particular, if $RA=I$, then $B$ is a left inverse of $A$ satisfying
these support identities.
\end{lemma}

\begin{proof}
For $y\in(\sum_{i\leq m}u_i)C_E(H)(\sum_{j\leq m}v_j)$, independent
row and column signs give
\[
 By=\E_{\epsilon,\tau}
       d_\epsilon R(U_\epsilon yV_\tau)d_\tau
   =\sum_{i,j\leq m}(R(u_i yv_j))_{ij}e_{ij}.
\]
Here
$U_\epsilon=I_H+\sum_{i\leq m}(\epsilon_i-1)u_i$ and
$V_\tau=I_H+\sum_{j\leq m}(\tau_j-1)v_j$, while
$d_\epsilon e_i=\epsilon_i e_i$ and $d_\tau e_i=\tau_i e_i$ for
$i\leq m$, with both diagonal operators equal to the identity on
$\operatorname{span}\{e_i:i>m\}$.
Expansion cancels every output coefficient outside
$\{1,\ldots,m\}\times\{1,\ldots,m\}$,
and the average has norm at most $\norm{R}$. The formulas agree on nested
corners. These corners are dense in $uC_E(H)v$: their support projections
converge strongly, hence in ideal norm on each finite-rank operator, and
finite-rank approximation gives convergence on the ideal. Thus $B$
extends to $uC_E(H)v$; precomposition with $y\mapsto uyv$ defines $B$ on
$C_E(H)$. The coordinate identities and $BA=M_\lambda$ follow
first on finite corners and then by continuity.

We will also use the row version of this construction. Averaging only
the row signs gives
\[
 B_{\mathrm{row}}y=\sum_i e_{ii}R(u_i y),
 \qquad \norm{B_{\mathrm{row}}}\leq\norm{R},
 \qquad B_{\mathrm{row}}(u_i y)=e_{ii}B_{\mathrm{row}}y.
\]
The same density argument defines the sum on $uC_E(H)$ and then on
$C_E(H)$ by precomposition with $y\mapsto uy$. For every $j$,
$(B_{\mathrm{row}}y)_{ij}=(R(u_i y))_{ij}$. The column version follows by using right signs.
\end{proof}

\begin{lemma}\label{pg:schur-stabilization}
Let $M_\lambda$ be a bounded Schur multiplier on $C_F$, and suppose that
for all $i>j$ with $i,j$ in an infinite subset $I\subset\mathbb N$ one has
$\abs{\lambda_{ij}}\geq\delta>0$.

If $\lambda_{ij}=0$ for $i\leq j$, then $t_F<\infty$.
If instead $t_F=\infty$, then for every $0<\varepsilon<\delta$ there are
disjoint infinite row and column sets, an isometric rectangular insertion
$J:C_F\to C_F$, its contractive extraction $C$, and $\mu\in\K$ with
$\abs{\mu}\geq\delta$ such that
\[
 CM_\lambda J=\mu I+E_0,\qquad \norm{E_0}\leq\varepsilon.
\]
This operator is invertible, with inverse norm at most
$(\delta-\varepsilon)^{-1}$.
\end{lemma}

\begin{proof}
Put $K=\norm{M_\lambda}$; testing units gives $\abs{\lambda_{ij}}\leq K$.
For the first assertion fix $N\geq2$. Color lower values on $I$ by a finite
partition into sets of diameter $\delta/(4N^2)$, and use the pair Ramsey
argument to find $N$ increasing indices all of whose lower values lie in
one set. Choose one of those values $\mu$, so $\abs{\mu}\geq\delta$. On this
corner the multiplier differs from $\mu L_F$ by an operator of norm at most
$\delta/4$: sum the at most $N^2$ errors, using
$\abs{x_{ij}}\leq\norm{x}_{C_F}$ and $\norm{e_{ij}}_{C_F}=1$. The compression to these $N$ indices is isometric to $C_F(\K^N)$. Therefore
\[
 \norm{L_F|_{M_N}}\leq K/\delta+1/4
\]
uniformly in $N$. Density proves boundedness of $L_F$.

For the second assertion choose $\eta_l\downarrow0$ with
$4\sum_l\eta_l\leq\varepsilon$. Successive finite Ramsey selections for
the bounded pairs $(\lambda_{qp},\lambda_{pq})$, $p<q$, give nested
infinite sets $I_l$ and increasing $n_l\in I_l$ such that all later $n_k$
belong to $I_l$, and all values on pairs from $I_l$ lie within $\eta_l$
of a center $(\mu_l,\nu_l)$. Centers at stages $k\geq l$ differ by at most
$\eta_l+\eta_k$ in each coordinate, since $I_k\subseteq I_l$ for $k\geq l$. They converge to $(\mu,\nu)$, and, for $i<j$,
\[
 \abs{\lambda_{n_jn_i}-\mu}\leq2\eta_i,
 \qquad \abs{\lambda_{n_in_j}-\nu}\leq2\eta_i,
 \qquad \abs{\mu}\geq\delta.
\]
Compress $M_\lambda$ to the matrix indices $(n_l)$ and relabel
$n_l$ as $l$. After removing the diagonal multiplier, whose norm is
at most $K$, the off-diagonal multiplier is
\[
 \mu L_F+\nu(I-\mathcal D_F-L_F)+E.
\]
The error has norm at most $4\sum_i\eta_i$. Indeed, for each smaller
index $i$, its upper errors form a single row multiplier
$x\mapsto e_{ii}x d_i$, with $\norm{d_i}\leq2\eta_i$, and its
lower errors form a column multiplier $x\mapsto d'_i x e_{ii}$ with the
same bound. This series converges absolutely in operator norm and
accounts for every off-diagonal error. If $\mu\ne\nu$, the displayed
identity would make $(\mu-\nu)L_F$ bounded. Thus $\mu=\nu$.

In the ideal compressed to $(n_l)$ and reindexed by $l$,
use the insertion $e_{ij}\mapsto e_{2i,2j-1}$ and its reverse
rectangular extraction. None of these entries is diagonal.
Combining these maps with insertion and extraction along $(n_l)$ gives the
claimed $J,C$ and $CM_\lambda J=\mu I+E_0$, with
$\norm{E_0}\leq\varepsilon$. Its norm-convergent Neumann series proves the
inverse bound. The inverse commutes with every source row and column
projection because the multiplier itself does.
\end{proof}

\begin{lemma}\label{pg:matched-component-retraction}
In Lemma~\ref{pg:matched-return}, suppose $t_F=\infty$ and
$\abs{(RA(e_{ij}))_{ij}}\geq\delta>0$ on all lower pairs of an infinite set.
Then one infinite rectangular source restriction $A_0=AJ$ has a bounded
left inverse $R_0$ satisfying
\[
 \norm{A_0}\leq\norm{A},\qquad \norm{R_0}\leq2\norm{R}/\delta.
\]
Writing $J(e_{ij})=e_{r_i c_j}$, one has
$A_0(e_{ij})=u_{r_i}A_0(e_{ij})v_{c_j}$ and, for every $y\in C_E(H)$,
\[
 R_0(u_{r_i}y)=e_{ii}R_0(y),\qquad
 R_0(yv_{c_j})=R_0(y)e_{jj}.
\]
\end{lemma}

\begin{proof}
Use $B$ from Lemma~\ref{pg:matched-return}, and apply
Lemma~\ref{pg:schur-stabilization} with $\varepsilon=\delta/2$ to $BA$.
Set $K=CM_\lambda J$ and $R_0=K^{-1}CB$. Then $R_0A_0=I$, with the stated
bounds. The support identities of $B$ from Lemma~\ref{pg:matched-return},
the rectangular extraction, and
commutation of $K^{-1}$ with row and column projections prove the last
assertion.
\end{proof}

\begin{lemma}\label{pg:finite-inner-head}
Let $H_0$ be a separable infinite-dimensional Hilbert space. Suppose
$A:C_F\to C_E(H_0\otimes\ell_2)$ is bounded and
$A(e_{ij})=b_{ij}\otimes e_{ij}$, and let $R:C_E(H_0\otimes\ell_2)\to C_F$
be any bounded map. Put $M=\norm{A}$ and $r=\norm{R}$. For an orthogonal
projection $p$ on $H_0$ of finite rank $d$, set
\[
 \mu_{ij}^p=(R((p\otimes I)A(e_{ij})))_{ij}.
\]
If there are pairs $(u_l,v_l)$ whose row indices are pairwise distinct and
whose column indices are pairwise distinct, and
$\abs{\mu_{u_lv_l}^p}\geq\delta>0$ for all $l$, then
\[
 \norm{a}_E\leq\frac{dMr}{\delta}\norm{a}_F,
 \qquad
 \norm{a}_F\leq\frac{dMr}{\delta}\norm{a}_E
 \quad(a\in c_{00}).
\]
The same assertion holds for right compression. Consequently, if $E$ and
$F$ are not canonically equal, then for each fixed finite-rank $p$,
\begin{equation}\label{pg:head-cofinal}
 \lim_{N\to\infty}\sup_{i>N,\,j>N}\abs{\mu_{ij}^p}=0,
\end{equation}
and the corresponding right-compression estimate holds as well.
\end{lemma}

\begin{proof}
Apply Lemma~\ref{pg:matched-return} to $A$ with
$u_i=I_{H_0}\otimes e_{ii}$ and $v_j=I_{H_0}\otimes e_{jj}$.
The resulting map $B$, of norm at most $r$, satisfies
$B(z\otimes e_{ij})=(R(z\otimes e_{ij}))_{ij}e_{ij}$.
Put $y_l=(pb_{u_lv_l})\otimes e_{u_lv_l}$ and $Ya=\sum_l a_ly_l$.
The operators $y_l$ have pairwise orthogonal left supports and pairwise
orthogonal right supports, rank at most $d$, and norm at most $M$.
The source $\sum_l a_le_{u_lv_l}$ has singular values $\abs{a_l}$, so
\[
 \norm{Ya}_{C_E}\leq M\norm{a}_F.
\]
The formula for $B$ gives
$B(Ya)=\sum_l a_l\mu_{u_lv_l}^p e_{u_lv_l}$. Solidity of $F$ therefore
gives
\begin{equation}\label{pg:head-return}
 \frac{\delta}{r}\norm{a}_F\leq\norm{Ya}_{C_E}\leq M\norm{a}_F.
\end{equation}
In particular $\norm{y_l}_{C_E}\geq\delta/r$. Since a rank at most $d$
operator has ideal norm at most $d$ times its largest singular value,
\[
 \frac{\delta}{rd}\leq s_1(y_l)\leq M.
\]
The singular values of $Ya$ are the disjoint union of
$\abs{a_l}s_t(y_l)$, $1\leq t\leq d$, with zero padding. Keeping the largest
one from each $l$, and then splitting the full union into its $d$ levels,
respectively yield
\[
 \frac{\delta}{rd}\norm{a}_E\leq\norm{Ya}_{C_E}\leq dM\norm{a}_E.
\]
Combining these bounds with \eqref{pg:head-return} proves both comparisons.
For right compression, $b_{u_lv_l}p$ has rank at most $d$ and
$B(Ya)$ has the same coordinate formula with the right-compression
coefficients, so the calculation applies unchanged.

If \eqref{pg:head-cofinal} fails, choose $\delta>0$ and successively choose
pairs beyond every previously used row and column index with coefficient
modulus at least $\delta$. The resulting norm comparisons imply $E=F$,
a contradiction. Equality of the sequence sets follows from the uniform
bounds on $c_{00}$ as explained in Section~\ref{sec:preliminaries}.
\end{proof}

\begin{theorem}\label{pg:matched-dichotomy}
Suppose $A:C_F\to C_E(H_0\otimes\ell_2)$ and
$R:C_E(H_0\otimes\ell_2)\to C_F$ are bounded, $RA=I$, and
\[
 A(e_{ij})=b_{ij}\otimes e_{ij}.
\]
Assume that there are finite-dimensional inner subspaces $\mathcal L_i$
and $\mathcal D_j$ such that
\[
 b_{ij}=p_{\mathcal L_i}b_{ij}p_{\mathcal D_j}
 \qquad(i,j\geq1).
\]
Their dimensions may be unbounded, and neither family is required to be
orthogonal or equal to the other family. Then either $E=F$ canonically,
or $F=\ell_2$ canonically.
\end{theorem}

\begin{proof}
Lemma~\ref{pg:matched-return} replaces $R$ by a left inverse of no larger
norm satisfying, on $C_E(H_0\otimes\ell_2)$,
\begin{equation}\label{pg:matched-beta}
 R(z\otimes e_{ij})=\beta_{ij}(z)e_{ij}.
\end{equation}
It still satisfies $RA=I$. Write $M=\norm{A}$, $r=\norm{R}$, and choose
$0<\eta<1/r$.

Put $d_j=\dim\mathcal D_j$,
$l_i=\dim\mathcal L_i$, $s_m=\sum_{j\leq m}d_j$ and
$t_m=\sum_{i\leq m}l_i$. For each $i$, the nested subspaces
\[
 X_{i,m}=\operatorname{span}\{\operatorname{ran}b_{ij}:j\leq m\}
 \subseteq\mathcal L_i
\]
have dimensions at most $s_m$. Successive Gram--Schmidt orthogonalization of the images under $b_{ij}$
of bases of $\mathcal D_j$ gives a basis adapted to these subspaces.
Complete it in $\mathcal L_i$, and let a surjective isometry $U_i$ send it to the first
$l_i$ vectors of a fixed orthonormal basis $(f_k)$ of $H_0$. Then
$U_iX_{i,m}\subseteq\operatorname{span}\{f_1,\ldots,f_{s_m}\}$.
In the same way, from the nested spans of
$\operatorname{ran}b_{ij}^*$, $i\leq m$, choose surjective isometries $V_j$ such that
\[
 V_j\mathcal D_j=\operatorname{span}\{f_1,\ldots,f_{d_j}\},
 \quad
 V_j\operatorname{span}\{\operatorname{ran}b_{ij}^*:i\leq m\}
 \subseteq\operatorname{span}\{f_1,\ldots,f_{t_m}\}.
\]
Right and left multiplication preserve,
respectively, $\operatorname{ran}b_{ij}$ and
$\operatorname{ran}b_{ij}^*$. Hence the
two choices can be made independently. Set $U=\diag(U_i)$, $V=\diag(V_j)$,
and replace $A,R$ by $UA(\cdot)V^*$ and $R(U^*(\cdot)V)$. Norms,
\eqref{pg:matched-beta}, and the identity $RA=I$ are preserved.

Call the new coefficients $b'_{ij}$. For fixed $k$, every $b'_{ik}$ maps
the fixed $d_k$-dimensional coordinate space into the fixed
$s_k$-dimensional coordinate space; every $b'_{kj}$ maps the fixed
$t_k$-dimensional coordinate space into the fixed $l_k$-dimensional one.
These are bounded arrays in finite-dimensional
matrix spaces. Operator-norm convergence there implies ideal-norm
convergence, since a rank at most $q$ difference has norm at most $q$
times its operator norm. A countable diagonal subsequence therefore gives
one infinite set $S\subset\mathbb N$ such that, for every fixed
$k\in\mathbb N$,
\[
 b'_{nk}\longrightarrow B^{(k)},\qquad
 b'_{kn}\longrightarrow C^{(k)}
 \quad(n\to\infty,\ n\in S)
\]
in ideal norm, with finite inner matrices as limits.

Choose positive $\epsilon_{rs}$, $r>s$, with
$2\sum_{r>s}\epsilon_{rs}<\eta$. Recursively select increasing $n_r\in S$
so that for every $s<r$ both
\[
 \norm{b'_{n_rn_s}-B^{(n_s)}}_{C_E}<\epsilon_{rs},\qquad
 \norm{b'_{n_sn_r}-C^{(n_s)}}_{C_E}<\epsilon_{rs}
\]
hold. Only finitely many previously fixed indices occur at each step.
Let $J(e_{ij})=e_{n_i n_j}$ and let $C$ be its reverse
coordinate extraction on $C_F$. Let
$\Phi(z\otimes e_{ij})=z\otimes e_{n_i n_j}$ and let $\Psi$ be
its reverse compression on $C_E(H_0\otimes\ell_2)$. Calling the transformed pair
$A',R'$, set
\[
 \widetilde A=\Psi A'J,\qquad \widetilde R=CR'\Phi.
\]
The identity $A'(e_{ij})=b'_{ij}\otimes e_{ij}$ places $A'J$ in the
corner with outer row and column indices in $\{n_r:r\geq1\}$,
so $\widetilde R\widetilde A=I$, with norms
at most $M,r$.

Set $B_s=B^{(n_s)}$, $C_s=C^{(n_s)}$, and $D_s=b'_{n_sn_s}$, and define
on finite matrices
\begin{equation}\label{pg:stable-profiles}
 A_0e_{nk}=B_k\otimes e_{nk},\quad
 A_0e_{kn}=C_k\otimes e_{kn}\quad(n>k),
 \qquad A_0e_{kk}=D_k\otimes e_{kk}.
\end{equation}
The coefficient maps $x\mapsto x_{rs}(A_0-\widetilde A)e_{rs}$ have norms
whose sum is less than $\eta$. They therefore define an operator
$E_0=A_0-\widetilde A$ with $\norm{E_0}<\eta$. Define
\[
 R_0=(I+\widetilde R E_0)^{-1}\widetilde R,
 \qquad M_0=M+\eta,
 \qquad r_0=\frac{r}{1-r\eta}.
\]
The Neumann series gives $R_0A_0=I$, $\norm{A_0}\leq M_0$ and
$\norm{R_0}\leq r_0$. The operator $\widetilde R E_0$ is a coordinate
multiplier, so it and its Neumann inverse commute with all source row and
column projections. Thus $R_0$ satisfies \eqref{pg:matched-beta} on the
whole target.

Define, initially on finitely supported vectors,
\[
 W\alpha=\sum_k\alpha_k B_k\otimes e_{1k}.
\]
For $n$ larger than the support of $\alpha$, this operator is obtained by
moving the single outer row of $A_0(\sum_k\alpha_ke_{nk})$ to row $1$.
Consequently $\norm{W\alpha}_{C_E}\leq M_0\norm{\alpha}_2$, and $W$ extends to
$\ell_2$. Let $c_n:C_F\to\ell_2$ extract the $n$th row, a contraction,
and put
\[
 H_ny=c_nR_0((I\otimes e_{n1})y).
\]
Then $\norm{H_n}\leq r_0$ and $H_nWe_k=e_k$ for $n>k$.
Separability of $C_E(H_0\otimes\ell_2)$ and weak sequential compactness of bounded
subsets of $\ell_2$ give a subsequence $n_t\to\infty$ for which $H_{n_t}y$
converges weakly on a countable dense set. The uniform bound extends
this convergence to every $y$ by norm approximation. The limits define a
bounded linear map $H$ satisfying
\[
 \norm{H}\leq r_0,\qquad HW=I_{\ell_2}.
\]
Equation~\eqref{pg:matched-beta} implies that each $H_n$ sends outer
column $k$ into coordinate $k$; so does $H$. Therefore the functionals
\[
 \beta_k(z)=(H(z\otimes e_{1k}))_k
\]
have norm at most $r_0$ and satisfy $\beta_k(B_k)=1$.

Assume now that $E$ and $F$ are not canonically equal. Let $p_m$ project
onto the first $m$ inner coordinates. Lemma~\ref{pg:finite-inner-head},
applied to $A_0,R_0$, gives \eqref{pg:head-cofinal}. For $n>k$,
\[
 (H_n((p_mB_k)\otimes e_{1k}))_k
   =(R_0((p_mB_k)\otimes e_{nk}))_{nk}=\mu_{nk}^{p_m}.
\]
For any fixed $m$, it follows that
\begin{equation}\label{pg:profile-head-decay}
 \beta_k(p_mB_k)\longrightarrow0\quad(k\to\infty).
\end{equation}
Indeed, given $\varepsilon>0$, \eqref{pg:head-cofinal} supplies $N$ such that
every $k>N$ and every sufficiently late $n_t>\max(k,N)$ give a scalar of
modulus at most $\varepsilon$. Passing to the weak limit for this fixed
$k$ preserves that bound. The cutoff $N$ may depend on $m$.

Starting with $m_0=0$, use \eqref{pg:profile-head-decay} to choose increasing
$k_l$ such that
$\abs{\beta_{k_l}(p_{m_{l-1}}B_{k_l})}\leq1/4$. Since $B_{k_l}$ is a finite
inner matrix, choose $m_l>m_{l-1}$ with $p_{m_l}B_{k_l}=B_{k_l}$. Put
\[
 h_l=p_{m_l}-p_{m_{l-1}},\qquad
 c_l=h_lB_{k_l},\qquad \gamma_l=\beta_{k_l}(c_l).
\]
Then the final supports of the $c_l$ are pairwise orthogonal and
\begin{equation}\label{pg:band-return}
 \abs{\gamma_l-1}\leq1/4,\qquad \abs{\gamma_l}\geq3/4.
\end{equation}
The formula
\[
 \mathcal Py=\sum_l(h_l\otimes I)y(I\otimes e_{k_lk_l})
\]
defines a contraction, with convergence in ideal norm, by
\eqref{eq:orthogonal-projections}: the projections $(h_l\otimes I)_l$
and $(I\otimes e_{k_lk_l})_l$ are separately pairwise orthogonal.

For a finite scalar vector $a=(a_1,\ldots,a_q)$, put
\[
 Q(a)=\norm{\sum_{l\leq q}a_lc_l\otimes e_{1k_l}}_{C_E}.
\]
The operators
$\sum_l a_lc_l\otimes e_{nk_l}$ with one common row, and
$\sum_l a_lc_l\otimes e_{n_lk_l}$ with pairwise distinct rows, have the
same singular values. In the first case their squared absolute value is
block diagonal with blocks $\abs{a_l}^2c_l^*c_l$ because $c_l^*c_s=0$ for
$l\ne s$; in the second case outer row orthogonality gives the same
identity. Thus both norms equal $Q(a)$.

Choosing all the rows greater than $k_q$ and using
\eqref{pg:stable-profiles}, contractivity of $\mathcal P$, and boundedness
of $A_0$ yields
\begin{equation}\label{pg:two-upper-comparisons}
 Q(a)\leq M_0\norm{a}_2,\qquad Q(a)\leq M_0\norm{a}_F.
\end{equation}
The matrix $\sum_{l\leq q}a_le_{nk_l}$ has rank at most one;
$\sum_{l\leq q}a_le_{n_lk_l}$ is a weighted partial permutation with
singular values $\abs{a_l}$. Since $H$
sends outer column $k$ to coordinate $k$, \eqref{pg:band-return} gives
\begin{equation}\label{pg:hilbert-profile-lower}
 \frac34\norm{a}_2
 \leq\norm{H\left(\sum_{l\leq q}a_lc_l\otimes e_{1k_l}\right)}_2
 \leq r_0Q(a).
\end{equation}
For each of these finitely many $l$, the definition of the same weak
limit $H$ gives
\[
 (R_0(c_l\otimes e_{n_tk_l}))_{n_tk_l}\longrightarrow\gamma_l.
\]
Choose pairwise distinct $n_l$ from this one subsequence, all exceeding
$k_q$, for which the corresponding scalars $\theta_l$ satisfy
$\abs{\theta_l-\gamma_l}\leq1/4$. Equation~\eqref{pg:matched-beta} gives
\[
 R_0\left(\sum_{l\leq q}a_lc_l\otimes e_{n_lk_l}\right)
    =\sum_{l\leq q}a_l\theta_le_{n_lk_l}.
\]
As $\abs{\theta_l}\geq1/2$, solidity implies
\begin{equation}\label{pg:parameter-profile-lower}
 \frac12\norm{a}_F\leq r_0Q(a).
\end{equation}
The quantity $Q(a)$ is independent of these row choices. Combining
\eqref{pg:two-upper-comparisons}--\eqref{pg:parameter-profile-lower} proves
\[
 \norm{a}_F\leq2r_0M_0\norm{a}_2,\qquad
 \norm{a}_2\leq\tfrac43r_0M_0\norm{a}_F
 \quad(a\in c_{00}).
\]
The bounds hold for every finite $a$ with the same constants, and hence
$F=\ell_2$ by density and coordinate continuity.
\end{proof}

\begin{theorem}\label{pg:unbounded-uniqueness}
If $t_F=\infty$ and $C_F$ embeds complementably into $C_E$, then $E=F$
canonically. Consequently an isomorphism $C_E\simeq C_F$, with triangular
truncation unbounded on either ideal, forces canonical equality.
\end{theorem}

\begin{proof}
By Lemmas~\ref{pg:full-reproduction} and~\ref{pg:coarse-components},
there is a component $C_{ab}$ for which
\[
 \abs{(RC_{ab}(e_{ij}))_{ij}}\geq1/5
\]
on all lower pairs of an infinite source set. Apply
Lemma~\ref{pg:matched-return} with $A=I_{C_F}$, the bounded map $RC_{ab}$,
and scalar coordinate projections. It gives a bounded Schur multiplier
with coefficients $(RC_{ab}(e_{ij}))_{ij}$. For any of the lower-only components
$C_{11},C_{10},C_{01}$ the multiplier vanishes on the upper triangle and
diagonal. Lemma~\ref{pg:schur-stabilization} would then imply $t_F<\infty$,
a contradiction. The selected component is therefore $C_{12}$ or $C_{21}$.

Transposing the target of $C_{21}$ and precomposing $R$ with transpose
preserves the coefficients $(RC_{21}(e_{ij}))_{ij}$ and puts its values
in $q_iC_Eq_j$. Transposition is linear and isometric.
Thus either component gives a map $A$ satisfying
$A(e_{ij})\in q_iC_Eq_j$. By Lemma~\ref{pg:matched-component-retraction},
a restriction $AJ$, with $J(e_{ij})=e_{r_i c_j}$, has a bounded
left inverse and satisfies $AJ(e_{ij})\in q_{r_i}C_Eq_{c_j}$.
Denote $AJ$ again by $A$ and its left inverse by $R_0$.

Choose Hilbert isometries $J_L,J_R$ which place $q_{r_i}\ell_2$ and
$q_{c_j}\ell_2$ in $H_0\otimes e_i$ and $H_0\otimes e_j$, respectively. Since each is
finite dimensional, the remaining Hilbert-space complements can be embedded
in unused inner coordinates. Replace the forward map by
$x\mapsto J_LA(x)J_R^*$ and its left inverse by
$y\mapsto R_0(J_L^*yJ_R)$. These transformations preserve singular values,
boundedness, and the left-inverse identity. The transformed coefficients
have finite-dimensional inner left supports depending only on $i$ and
right supports depending only on $j$.

Theorem~\ref{pg:matched-dichotomy} therefore gives either $E=F$ canonically
or $F=\ell_2$ canonically. The
latter would make the norm of $C_F$ equivalent to the Hilbert--Schmidt
norm, in which $L_F$ is a contraction. Equivalence of norms would make
$L_F$ bounded also on $C_F$, contradicting the hypothesis. Thus $E=F$.
For an onto isomorphism orient the map with an ideal having unbounded
triangular truncation as its source and use its inverse as left inverse.
\end{proof}

Let $X_-(E)$ be the norm closure in $C_E$ of the finite strictly lower
triangular matrices, and define $X_-(F)$ in the same way.
If $t_E<\infty$, corner approximation gives
$X_-(E)=L_E(C_E)$.

\begin{lemma}\label{pg:hilbert-head}
Let $F$ be a normalized separable symmetric sequence space, and suppose that
$B:C_F\to C_F$ has a factorization $B=CA$ through a Hilbert space $H$.
Put $a=\norm{A}$, $b=\norm{C}$, and $\lambda_{ij}=(Be_{ij})_{ij}$.
The Schur map $M_\lambda e_{ij}=\lambda_{ij}e_{ij}$ factors through a Hilbert
space with factorization product at most $ab$.

If $(u_l)$ and $(v_l)$ are separately pairwise distinct and
$\abs{\lambda_{u_l v_l}}\geq\delta>0$ for every $l$, then
\begin{equation}\label{pg:hilbert-head-comparison}
 \frac{\delta}{ab}\norm{c}_2\leq\norm{c}_F
       \leq\frac{ab}{\delta}\norm{c}_2
 \qquad(c\in c_{00}).
\end{equation}
Consequently, if $F$ is not canonically $\ell_2$, then
\begin{equation}\label{pg:hilbert-head-cofinal}
 \lim_{N\to\infty}\sup_{i,j>N}\abs{(Be_{ij})_{ij}}=0.
\end{equation}

In particular, let $V:C_F\to C_E(H_1)$ and $R:C_E(H_1)\to C_F$ be bounded,
and let $p$ be a finite-rank Hilbert projection, of rank $d\geq1$.
The operator $x\mapsto R((Vx)p)$ has a Hilbert factorization with product at
most $d\norm{V}\norm{R}$. The same assertion holds for $x\mapsto R(pVx)$, and both
operators satisfy \eqref{pg:hilbert-head-cofinal} when $F$ is not canonically
$\ell_2$.
\end{lemma}

\begin{proof}
Let $\Omega=\{-1,1\}^{\mathbb N}\times\{-1,1\}^{\mathbb N}$ with its product
probability measure, and write $d_\varepsilon,d_\eta$ for the corresponding
diagonal sign operators. The map
$\omega\mapsto d_\varepsilon x d_\eta$ is norm continuous for every $x\in C_F$:
this holds for finite matrices, and their norm approximations are uniform in
the signs. Replace $H$ by the separable closed span of $A(C_F)$ and define
\[
 \mathcal Ux(\omega)=A(d_\varepsilon x d_\eta),\qquad
 \mathcal Qh=\int_\Omega d_\varepsilon C(h(\omega))d_\eta\,d\omega.
\]
Then $\mathcal U:C_F\to L_2(\Omega;H)$ has norm at most $a$.
The second integral is a Bochner integral in $C_F$. For simple $h$ its strong
measurability follows from the continuous sign action; approximation by
simple functions gives it for arbitrary $h\in L_2(\Omega;H)$. Moreover,
\[
 \int_\Omega\norm{d_\varepsilon C(h(\omega))d_\eta}_{C_F}\,d\omega
 \leq b\norm{h}_{L_1(H)}\leq b\norm{h}_{L_2(H)}.
\]
Thus $\norm{\mathcal Q}\leq b$. On a matrix unit, its $(k,l)$ coefficient is
\[
 (\mathcal Q\mathcal Ue_{ij})_{kl}
   =(Be_{ij})_{kl}\E(\varepsilon_i\varepsilon_k\eta_j\eta_l)
   =\lambda_{ij}\,\mathbf1_{\{(k,l)=(i,j)\}}.
\]
Continuous coordinate functionals pass through the Bochner integral. Hence
$M_\lambda=\mathcal Q\mathcal U$ on the entire ideal by density.

For the second assertion, the map
$Jc=\sum_l c_l e_{u_l v_l}$ is an isometry $F\to C_F$.
Its coordinate extraction $D:C_F\to F$ is contractive, by row and column
compression followed by $\mathcal D_F$; density gives its range in $F$. The map
$DM_\lambda J$ is the diagonal multiplier with coefficients
$\lambda_{u_l v_l}$; its inverse on $F$ has norm at most $\delta^{-1}$.
Therefore
\[
 T=\mathcal UJ,\qquad
 Z=(DM_\lambda J)^{-1}D\mathcal Q
\]
satisfy $ZT=I_F$, $\norm{T}\leq a$, and $\norm{Z}\leq b/\delta$.
In particular
$\delta/b\leq\norm{Te_l}_{L_2(\Omega;H)}\leq a$. For finite $c$, independent signs give
\[
 \E\norm{\sum_l\varepsilon_l c_l Te_l}_{L_2(\Omega;H)}^2
       =\sum_l\abs{c_l}^2\norm{Te_l}_{L_2(\Omega;H)}^2.
\]
The upper and lower bounds for $T$ apply to every signed vector, whose
$F$ norm equals $\norm{c}_F$. Comparing both sides also with
$(\delta/b)^2\norm{c}_2^2$ and $a^2\norm{c}_2^2$ proves
\eqref{pg:hilbert-head-comparison}.
If \eqref{pg:hilbert-head-cofinal} fails, a fixed $\delta>0$ and a recursive
selection beyond all previously chosen row and column indices give the
paired sequence just considered. Equation~\eqref{pg:hilbert-head-comparison}
would imply that $F$ is canonically $\ell_2$, a contradiction.

Finally, every operator $z=zp$ has rank at most $d$, and normalization gives
\[
 \norm{z}_2\leq\sqrt d\,\norm{z}_{C_E},\qquad
 \norm{z}_{C_E}\leq\norm{z}_1\leq\sqrt d\,\norm{z}_2.
\]
Factor $x\mapsto R((Vx)p)$ through the Hilbert space of Hilbert--Schmidt
operators with right support $p$. The two maps have norms at most
$\sqrt d\,\norm{V}$ and $\sqrt d\,\norm{R}$. Left support is treated in exactly
the same way, using the Hilbert space of Hilbert--Schmidt operators satisfying
$z=pz$.
\end{proof}

\begin{remark}
The limit in \eqref{pg:hilbert-head-cofinal} requires both matrix
indices $i,j$ to tend to infinity. In the compression assertion of
Lemma~\ref{pg:hilbert-head}, the projection $p$ on $H_1$ must have finite
rank. After isometric insertion of finitely many finite-dimensional
support spaces, use the projection onto the span of their images.
\end{remark}

\begin{lemma}\label{pg:interlacing}
If $t_E<\infty$, there are bounded maps
$J_E:C_E\to X_-(E)$ and $Q_E:X_-(E)\to C_E$ such that
\[
 Q_EJ_E=I,\qquad \norm{J_E}\leq2t_E,\qquad \norm{Q_E}\leq2.
\]
The same assertion holds with $E$ replaced by $F$.
\end{lemma}

\begin{proof}
Let $u_e e_i=e_{2i}$ and $u_o e_i=e_{2i-1}$. Set
\begin{equation}\label{pg:interlacing-maps}
 J_Ex=L_E(u_e x u_o^*-u_o x^{\mathsf t}u_e^*),\qquad
 Q_Ey=u_e^*yu_o-(u_o^*yu_e)^{\mathsf t}.
\end{equation}
Ideal contractivity and invariance under transpose give the stated bounds.
For a source unit,
\[
 J_Ee_{ij}=
 \begin{cases}
 e_{2i,2j-1},&i\geq j,\\
 -e_{2j-1,2i},&i<j.
 \end{cases}
\]
Substitution into $Q_E$ gives $Q_EJ_Ee_{ij}=e_{ij}$ in both cases, and
density proves $Q_EJ_E=I$.
\end{proof}

\begin{proposition}\label{pg:block-row}
Suppose $t_F<\infty$, and let $(H_i)$ be mutually orthogonal
finite-dimensional Hilbert subspaces. Let $Y$ be the closed linear span
in $C_E$ of the spaces $p_{H_i}C_Ep_{H_i}$.
Write $\mathcal P_i y=p_{H_i}yp_{H_i}$.
Let
\[
 A:X_-(F)\to Y,\qquad R:Y\to C_F,\qquad
 \norm{A}\leq M,\quad \norm{R}\leq r,
\]
and suppose $Ae_{ij}\in\mathcal P_iY$ whenever $i>j$. If, for an infinite
set $I\subset\mathbb N$ and $\delta>0$,
\begin{equation}\label{pg:block-row-returns}
 \abs{(RAe_{ij})_{ij}}\geq\delta\qquad(i>j, i,j\in I),
\end{equation}
then there are $T:C_F\to Y$ and $S:Y\to C_F$ with
\begin{equation}\label{pg:block-row-bounds}
 ST=I,\qquad \norm{T}\leq2t_FM,\qquad
 \norm{S}\leq\frac{4t_Fr}{\delta}.
\end{equation}
In particular $C_F$ embeds complementably into $S_E$.
\end{proposition}

\begin{proof}
Principal insertion and extraction along $I$ in $C_F$, and
compression of $Y$ to the closed span of
$p_{H_i}C_Ep_{H_i}$, $i\in I$, are contractive. Since
$Ae_{ij}\in p_{H_i}C_Ep_{H_i}$ for $i>j$, $i,j\in I$, the
forward range after these restrictions lies in that closed span.
The resulting maps preserve \eqref{pg:block-row-returns}.
Relabel to assume $I=\mathbb N$.
Extend $R$ to $C_E$ by composing it with the projection
$y\mapsto\sum_i p_{H_i}yp_{H_i}$ onto $Y$, and apply the row average
in Lemma~\ref{pg:matched-return}. Its restriction to $Y$ is
\[
 R_1y=\sum_i e_{ii}R(\mathcal P_i y),\qquad \norm{R_1}\leq r.
\]
It preserves \eqref{pg:block-row-returns} and sends
$p_{H_i}C_Ep_{H_i}$ into $e_{ii}C_F$.
Thus $L_FR_1$ has norm at most $t_Fr$ and sends
$p_{H_i}C_Ep_{H_i}$ into
$\operatorname{span}\{e_{ij}:j<i\}$.

Define $K_i\in\mathcal B(\ell_2)$ by
\[
 (K_i)_{kj}=(R_1Ae_{ij})_{ik}\quad(j,k<i),
 \qquad (K_i)_{kj}=0\quad\text{otherwise}.
\]
Single-row insertion and extraction give $\norm{K_i}\leq rM$. A diagonal
subsequence makes all matrix entries converge. Their limit defines a
sesquilinear form bounded by $rM$ on finite vectors, hence an operator
$K\in\mathcal B(\ell_2)$ with $\norm{K}\leq rM$. The uniform bound extends
the convergence to the weak operator topology. Denote its index set by $I_0$.
For every fixed $j$,
\begin{equation}\label{pg:block-row-limit-diagonal}
 d_j=K_{jj},\qquad \abs{d_j}\geq\delta.
\end{equation}

Choose positive sequences $\alpha_s,\beta_s$ with
\[
 \sum_s\alpha_s<\frac{\delta}{4t_F},\qquad
 \sum_s\beta_s<\frac{\delta}{4}.
\]
Select increasing $n_s\in I_0$ so that
\[
 \sum_{h<s}(\abs{K_{n_s n_h}}+\abs{K_{n_h n_s}})<\alpha_s
\]
and the sum of the Euclidean norms of the columns of
$((K_{n_s}-K)_{n_h n_j})_{h,j<s}$ is less than $\beta_s$.
The first condition is available because each fixed row and column of a
bounded Hilbert-space operator tends to zero at infinity. The second involves
only finitely many entries and follows from the weak operator convergence
along $I_0$. Thus both hold sufficiently far along $I_0$.

Let $K'$ be the compression of $K$ to $(n_s)$, reindexed on $\ell_2$, and put
$D=\diag(d_{n_s})$. Its off-diagonal entries are an absolutely summable series
of scalar matrix units, so
\begin{equation}\label{pg:block-row-first-error}
 \norm{K'-D}\leq\sum_s\alpha_s<\frac{\delta}{4t_F},\qquad
 \norm{D^{-1}}\leq\delta^{-1}.
\end{equation}
Let $J_0,C_0$ be principal source insertion and extraction along $(n_s)$,
and put
\[
 Y'=\overline{\sum_{s\geq1}p_{H_{n_s}}C_Ep_{H_{n_s}}}^{\,C_E}.
\]
Let $\Phi:Y'\to Y$ be inclusion and let
$\Psi y=\sum_s p_{H_{n_s}}yp_{H_{n_s}}$.
The range of $AJ_0$ lies in $Y'$. Define
\[
 A'=\Psi AJ_0,\qquad R'=C_0L_FR_1\Phi,\qquad B=R'A'.
\]
Then $\norm{A'}\leq M$, $\norm{R'}\leq t_Fr$, and row $s$ of $B$ has precisely the
matrix $\bigl((K_{n_s})_{n_k n_j}\bigr)_{k,j<s}$.

The comparison operator
\[
 B_0x=L_F\bigl(x(K')^{\mathsf t}\bigr),\qquad x\in X_-(F),
\]
is bounded by $t_F\norm{K'}$, since transposition preserves the norm of
bounded Hilbert-space operators. Its $(s,k)$ coefficient, $k<s$, is
$\sum_{j<s}K'_{kj}x_{sj}$, which verifies the transpose orientation.
On each unit $e_{sj}$ the difference $B-B_0$ occupies a single source row.
The second selection condition therefore gives
\[
 \sum_{j<s}\norm{(B-B_0)e_{sj}}_{C_F}<\beta_s.
\]
Every coordinate functional has norm at most one. Summing the associated
coefficient rank-one operators thus converges in operator norm and proves
\begin{equation}\label{pg:block-row-second-error}
 \norm{B-B_0}\leq\sum_s\beta_s<\delta/4.
\end{equation}

Right multiplication $B_Dx=xD$ is an automorphism of $X_-(F)$ with inverse
norm at most $\delta^{-1}$. Equations
\eqref{pg:block-row-first-error}--\eqref{pg:block-row-second-error} give
\[
 \norm{B-B_D}\leq\norm{B-B_0}+t_F\norm{K'-D}<\delta/2.
\]
The Neumann series on the closed space $X_-(F)$ therefore gives
$\norm{B^{-1}}\leq2/\delta$ and
$B^{-1}R'A'=I_{X_-(F)}$. With the maps in Lemma~\ref{pg:interlacing}, set
\[
 T=\Phi A'J_F,\qquad S=Q_FB^{-1}R'\Psi.
\]
Then $ST=I$, with the bounds \eqref{pg:block-row-bounds}.

Finally, choose strictly increasing integers $m_i$ with
$m_i\geq\dim H_i$. Insert $p_{H_i}C_Ep_{H_i}$ as a matrix corner of the
summand $M_{m_i}(\K)$ in $S_E$, and use reverse compression to those
corners. These maps preserve or decrease the
concatenated singular-value norm and have composite identity on $Y$.
They extend by finite-block density, proving the last assertion.
\end{proof}

\begin{proposition}\label{pg:lower-matched}
Suppose $t_E,t_F<\infty$ and put $Y=C_E(H_0\otimes\ell_2)$. Let
$A:X_-(F)\to Y$ and $R:Y\to C_F$ be bounded by $M$ and $r$, respectively.
Suppose
\begin{equation}\label{pg:lower-matched-form}
 Ae_{ij}=b_{ij}\otimes e_{ij},\qquad
 b_{ij}=p_{\mathcal L_i}b_{ij}p_{\mathcal D_j}\quad(i>j),
\end{equation}
where all $\mathcal L_i,\mathcal D_j\subset H_0$ are finite dimensional.
If $\abs{(RAe_{ij})_{ij}}\geq\delta>0$ on all lower pairs from one infinite
principal set, there are a map $V:C_F\to Y$ satisfying
$Ve_{ij}=v_{ij}\otimes e_{ij}$, where
$\operatorname{ran}v_{ij}$ lies in a finite-dimensional subspace of
$H_0$ depending only on $i$, and $\operatorname{ran}v_{ij}^*$ lies in
a finite-dimensional subspace of $H_0$ depending only on $j$, and a bounded map
$S:Y\to C_F$ such that
\begin{equation}\label{pg:lower-matched-bounds}
 SV=I,\qquad \norm{V}\leq4t_FM,\qquad
 \norm{S}\leq\frac{16t_Et_Fr}{\delta}.
\end{equation}
If instead $R:Y\to X_-(F)$ already satisfies $RA=I$, the bound for
$\norm{S}$ can be replaced by $8t_Er$. In either case,
Theorem~\ref{pg:matched-dichotomy} gives canonical $E=F$ or $F=\ell_2$.
\end{proposition}

\begin{proof}
Write $q_i=I_{H_0}\otimes e_{ii}$. Apply Lemma~\ref{pg:matched-return}
to $AL_F$ and $R$, with outer support projections $q_i$. It gives
$B:Y\to C_F$ satisfying
\[
 \norm{B}\leq r,\qquad (By)_{ij}=(R(q_i yq_j))_{ij}.
\]
Hence
\[
 BAe_{ij}=\lambda_{ij}e_{ij},\qquad
 \lambda_{ij}=(RAe_{ij})_{ij},\quad i>j.
\]
On the given principal set the coefficients have modulus at least $\delta$
and are bounded by $Mr$.

Choose positive $\eta_j$ with $\sum_j\eta_j<\delta/2$.
Successive finite-color infinite Ramsey selections give increasing indices
$n_j$ and $\mu\in\K$, $\abs{\mu}\geq\delta$, such that
\begin{equation}\label{pg:lower-matched-stabilization}
 \abs{\lambda_{n_i n_j}-\mu}\leq\eta_j\qquad(i>j).
\end{equation}
Use the nested selection in Lemma~\ref{pg:schur-stabilization}, now for
the lower coefficients alone: at stage $j$ make the diameter of their
value set at most $\eta_j$. The closures of these value sets are nested
compact sets, whose common point $\mu$ satisfies
\eqref{pg:lower-matched-stabilization} and $\abs{\mu}\geq\delta$.

Let $J_I(e_{ij})=e_{n_i n_j}$ and let $C_I$ be its reverse
coordinate extraction on $C_F$. Let
$\Phi(z\otimes e_{ij})=z\otimes e_{n_i n_j}$ and let $\Psi$ be
its reverse compression on $Y$.
Set $A_I=\Psi AJ_I$ and $B_I=C_IB\Phi$. By
\eqref{pg:lower-matched-form}, $AJ_I$ is supported in the selected
corner with outer row and column indices in $\{n_i:i\geq1\}$.
Thus $K=B_IA_I$ acts on $X_-(F)$
by the lower coefficients $\lambda_{n_i n_j}$.
For each column $j$ the error in \eqref{pg:lower-matched-stabilization} is a
map $x\mapsto d_jxe_{jj}$ with $\norm{d_j}\leq\eta_j$. The sum of these maps
converges in operator norm, whence
\[
 \norm{K-\mu I}\leq\sum_j\eta_j<\delta/2,\qquad
 \norm{K^{-1}}\leq2/\delta.
\]
This is a Neumann inverse on the complete space $X_-(F)$.
Consequently
\[
 R_I=K^{-1}L_FB_I:Y\to X_-(F),\qquad
 R_IA_I=I,\qquad \norm{R_I}\leq2t_Fr/\delta.
\]
The inverse commutes with source row and column coordinate projections,
since $K$ does. Consequently,
$R_I(q_iYq_j)\subseteq\K e_{ij}$ for $i>j$.
The coefficient of $A_I(e_{ij})$ has range in $\mathcal L_{n_i}$
and adjoint range in $\mathcal D_{n_j}$, so these support spaces
remain finite dimensional.

Now consider any pair $A:X_-(F)\to Y$,
$R:Y\to X_-(F)$ satisfying $RA=I$ and
\eqref{pg:lower-matched-form}. On $Y$ let $L_{\mathrm{out}}$ retain
$q_i yq_j$ precisely when $i>j$. Then
\begin{equation}\label{pg:outer-triangle-bound}
 \norm{L_{\mathrm{out}}}\leq2t_E.
\end{equation}
Indeed, for a finite coordinate matrix order the finite Hilbert basis first
by its outer label and then by its inner label. Ordinary strict triangular
truncation in that ordering has norm at most $t_E$. Removing its outer
diagonal blocks, by $I-\mathcal D_{\mathrm{out}}$ where
$\mathcal D_{\mathrm{out}}$ is the contractive projection
$y\mapsto\sum_i q_iyq_i$, gives exactly
$L_{\mathrm{out}}$. The bound is independent of the numbers of
inner and outer basis indices, so the consistent finite
formulas extend by norm density and prove \eqref{pg:outer-triangle-bound}.

Let $U_e=I_{H_0}\otimes u_e$ and $U_o=I_{H_0}\otimes u_o$, with $u_e,u_o$
as in \eqref{pg:interlacing-maps}, and define
\[
 \mathcal Ky=U_e^*yU_o-(U_o^*yU_e)^{\mathsf t},\qquad
 \mathcal Jz=L_{\mathrm{out}}
       (U_e zU_o^*-U_o z^{\mathsf t}U_e^*).
\]
Their norms are at most $2$ and $4t_E$. The same unit calculation as in
Lemma~\ref{pg:interlacing}, now allowing arbitrary inner coefficients,
gives $\mathcal K\mathcal J=I_Y$. Moreover,
$\mathcal J\mathcal K$ is the identity on matrices supported in
\begin{equation}\label{pg:parity-regions}
 (2i,2j-1), i\geq j,\qquad
 (2j-1,2i), i<j.
\end{equation}
For the second region, $b\otimes e_{2j-1,2i}$ is folded to
$-b^{\mathsf t}\otimes e_{ij}$ and unfolded back to itself. The two
transposes and signs cancel. The identities extend from finite matrices by
the bounds just proved.

Set
\[
 V=\mathcal KAJ_F,\qquad S=Q_FR\mathcal J.
\]
Equation~\eqref{pg:lower-matched-form} puts $AJ_F$ in
\eqref{pg:parity-regions}, so
\[
 SV=Q_FR\mathcal J\mathcal KAJ_F=Q_FRAJ_F=Q_FJ_F=I.
\]
The bounds are $\norm{V}\leq4t_FM$ and $\norm{S}\leq8t_Er$.
On a unit, the coefficient of $Ve_{ij}$ is
\[
 Ve_{ij}=
 \begin{cases}
 b_{2i,2j-1}\otimes e_{ij},&i\geq j,\\
 b_{2j-1,2i}^{\mathsf t}\otimes e_{ij},&i<j.
 \end{cases}
\]
The new inner left and right support spaces may be taken as
\[
 \mathcal L_{2i}+\overline{\mathcal D_{2i}},\qquad
 \mathcal D_{2j-1}+\overline{\mathcal L_{2j-1}},
\]
where the bar denotes coordinate conjugation.
Indeed $(p_Lbp_D)^{\mathsf t}=p_{\overline D}b^{\mathsf t}p_{\overline L}$.
The first space is finite dimensional and depends only on $i$;
the second is finite dimensional and depends only on $j$.
Applying this construction to $A_I,R_I$ proves
\eqref{pg:lower-matched-bounds}, and Theorem~\ref{pg:matched-dichotomy}
applies to $V,S$.
\end{proof}

\begin{proposition}\label{pg:one-sided}
Let $t_F<\infty$, suppose that $F$ is not canonically $\ell_2$, and let
$P_i$ be mutually orthogonal finite-rank projections on a separable Hilbert
space $H$. Suppose
\[
 A:X_-(F)\to C_E(H),\qquad R:C_E(H)\to C_F,\qquad
 \norm{A}\leq M,\quad\norm{R}\leq r,
\]
with
\begin{equation}\label{pg:one-sided-hypotheses}
 Ae_{ij}=P_iAe_{ij}\quad(i>j),\qquad
 \abs{(RAe_{ij})_{ij}}\geq\delta>0\quad(i>j, i,j\in I)
\end{equation}
for an infinite principal set $I$. Then one of the following holds:
\begin{enumerate}
\item There are maps $A_m:X_-(F)\to C_E(H_0\otimes\ell_2)$ and
$R_m:C_E(H_0\otimes\ell_2)\to C_F$, bounded by $M$ and $r$, respectively,
such that $A_m$ satisfies \eqref{pg:lower-matched-form}, with
finite-dimensional supports, and $\abs{(R_mA_me_{ij})_{ij}}$ has one fixed
positive lower bound for all $i>j$.
\item There are $T:C_F\to S_E$ and $S:S_E\to C_F$ such that
\[
 ST=I,\qquad \norm{T}\leq2t_FM,\qquad
 \norm{S}\leq8t_Fr/\delta.
\]
\end{enumerate}
If also $t_E<\infty$, the first alternative implies canonical $E=F$ by
Proposition~\ref{pg:lower-matched}. The same alternatives hold if the first
condition in \eqref{pg:one-sided-hypotheses} is replaced by
$Ae_{ij}=(Ae_{ij})P_i$.
\end{proposition}

\begin{proof}
Restrict $A$ to the closed span of $\{e_{ij}:i>j,\ i,j\in I\}$
and compress its values on the left by the strong sum $\sum_{i\in I}P_i$.
Compose $R$ with principal coordinate extraction along $I$. These
operations do not increase $\norm{A}$ or $\norm{R}$ and preserve
\eqref{pg:one-sided-hypotheses}.
Relabel so that $I=\mathbb N$.
By the row-average construction in Lemma~\ref{pg:matched-return}, replace
$R$ by the map
\begin{equation}\label{pg:one-sided-row-return}
 R_{\mathrm{row}}y=\sum_i e_{ii}R(P_i y).
\end{equation}
The series defines $R_{\mathrm{row}}$ on $(\sum_iP_i)C_E(H)$;
precomposition with $y\mapsto(\sum_iP_i)y$, where the projection sum
converges strongly, defines it on $C_E(H)$. Its norm is at most $r$,
and it preserves \eqref{pg:one-sided-hypotheses}. Denote it again by $R$.

Choose increasing finite-rank projections $Q_m\to I_H$ strongly,
with $Q_0=0$, and set, for $i>j$,
\[
 \lambda_{ij}=(RAe_{ij})_{ij},\qquad
 \lambda_{ij}^{(m)}=(R((Ae_{ij})Q_m))_{ij}.
\]
These scalars have modulus at most $rM$. For each fixed $m$,
\begin{equation}\label{pg:one-sided-fixed-head}
 \lim_{N\to\infty}\sup_{i>j>N}\abs{\lambda_{ij}^{(m)}}=0.
\end{equation}
To apply Lemma~\ref{pg:hilbert-head}, extend $A$ to $C_F$ as $AL_F$.
The operator $x\mapsto R((AL_Fx)Q_m)$ has a Hilbert factorization
with product at most
\[
 (\operatorname{rank}Q_m)t_FMr.
\]
Equation
\eqref{pg:hilbert-head-cofinal} therefore gives
\eqref{pg:one-sided-fixed-head}.

A countable diagonal selection gives an infinite set $\mathcal S$ on
which $\lambda_{nj}^{(m)}$ converges for every fixed pair $(j,m)$. Write
\begin{equation}\label{pg:one-sided-profiles}
 \nu_j^m=\lim_{\substack{n\in\mathcal S\\n\to\infty}}
                 \lambda_{nj}^{(m)},\qquad
 \nu_j^0=0,\qquad h_j=\sup_{m\geq0}\abs{\nu_j^m}.
\end{equation}
Only $n>j$ is relevant to each limit. For every fixed $m$,
\eqref{pg:one-sided-fixed-head} gives $\nu_j^m\to0$ as $j\to\infty$.
We distinguish two cases according to $\limsup_{j\in\mathcal S}h_j$.

Suppose first that this limsup is positive. Choose $a>0$ and an infinite
$\mathcal S_0\subset\mathcal S$ with $h_j>4a$ for $j\in\mathcal S_0$.
Recursively select increasing $j_l\in\mathcal S_0$ and $m_l$, with $m_0=0$.
When choosing $j_l$, impose
\[
 \abs{\nu_{j_l}^m}<a\quad(0\leq m\leq m_{l-1})
\]
and, for each $s<l$,
\[
 \abs{\lambda_{j_lj_s}^{(m_s)}-\nu_{j_s}^{m_s}}<a/2,\qquad
 \abs{\lambda_{j_lj_s}^{(m_{s-1})}-\nu_{j_s}^{m_{s-1}}}<a/2.
\]
For these finitely many fixed projection indices and earlier columns,
both requirements hold eventually in $\mathcal S_0$. Then choose $m_l$ with
$\abs{\nu_{j_l}^{m_l}}>4a$. Necessarily $m_l>m_{l-1}$ by the first inequality.
For $D_s=Q_{m_s}-Q_{m_{s-1}}$ this construction gives, for every $i>s$,
\begin{equation}\label{pg:one-sided-band-returns}
 \abs{(R((Ae_{j_i j_s})D_s))_{j_i j_s}}>2a.
\end{equation}
Indeed $\abs{\nu_{j_s}^{m_s}-\nu_{j_s}^{m_{s-1}}}>3a$, and
the approximation errors for $\lambda_{j_i j_s}^{(m_s)}$ and
$\lambda_{j_i j_s}^{(m_{s-1})}$ total less than $a$.

Let $J:X_-(F)\to X_-(F)$ be the principal insertion
$Je_{is}=e_{j_i j_s}$ for $i>s$. Define on finite lower matrices
\[
 A_bx=\sum_s A(J(xe_{ss}))D_s.
\]
For input in the first $N$ columns, this is the average of
$A(J(xd_\tau))$ multiplied on the right by the isometry with sign $\tau_s$
on $D_s$, $s\leq N$, and sign one on the complement. Expansion retains
precisely the terms with equal column and projection indices. Thus the
consistent finite formulas extend with norm at most $M$. Moreover
\[
 A_be_{is}=P_{j_i}A_be_{is}D_s\quad(i>s).
\]
Let $P$ and $D$ be the strong sums of $(P_{j_i})$ and $(D_s)$. Separate
Hilbert isometries $J_L:PH\to H_0\otimes\ell_2$ and
$J_R:DH\to H_0\otimes\ell_2$ satisfy
$J_L(P_{j_i}H)\subset H_0\otimes\operatorname{span}\{e_i\}$ and
$J_R(D_sH)\subset H_0\otimes\operatorname{span}\{e_s\}$.
Extend $J_L$ and $J_R$ by zero on $P^\perp H$ and $D^\perp H$,
respectively. Let $C:C_F\to C_F$ be the extraction
$(Cz)_{is}=z_{j_i j_s}$, and define
\[
 A_mx=J_LA_bxJ_R^*,\qquad
 R_my=C R(J_L^*yJ_R).
\]
The norms of $A_m$ and $R_m$ are at most $M$ and $r$, respectively.
The finite inner support spaces of $A_me_{is}$ in
\eqref{pg:lower-matched-form} depend only on $i$ on the left and only
on $s$ on the right. The maps $A_m,R_m$ preserve
\eqref{pg:one-sided-band-returns}, proving the first alternative.

Suppose now that $h_j\to0$ along $\mathcal S$. Choose increasing
$j_i\in\mathcal S$ and $m_i$, beginning with $m_0=0$. First choose $j_i$ so
that
\begin{equation}\label{pg:one-sided-small-profile}
 h_{j_i}\leq\delta/8,\qquad
 \abs{\lambda_{j_i j_s}^{(m_{i-1})}-\nu_{j_s}^{m_{i-1}}}
       \leq\delta/8\quad(s<i).
\end{equation}
The projection index $m_{i-1}$ and the earlier columns are fixed before
this choice;
the conditions are consequently compatible convergence requirements along
$\mathcal S$. The previously selected numbers $h_{j_s}$ also satisfy
$h_{j_s}\leq\delta/8$, so
\begin{equation}\label{pg:one-sided-old-head}
 \abs{\lambda_{j_i j_s}^{(m_{i-1})}}\leq\delta/4\quad(s<i).
\end{equation}
After fixing $j_i$, choose $m_i>m_{i-1}$ with
\begin{equation}\label{pg:one-sided-tail}
 \norm{(Ae_{j_i j_s})(I-Q_{m_i})}_{C_E}\leq\delta/(4r)
       \quad(s<i).
\end{equation}
Here $r>0$. Since $Q_m\to I$ strongly, $zQ_m\to z$ in ideal norm by
finite-rank approximation. Applying this convergence to
$Ae_{j_i j_s}$, $s<i$, proves that \eqref{pg:one-sided-tail} can be imposed.

For the sequence $(j_i)$ selected in this case, let
$Je_{is}=e_{j_i j_s}$ for $i>s$ and $(Cz)_{is}=z_{j_i j_s}$.
Set $D_i=Q_{m_i}-Q_{m_{i-1}}$ and define
\[
 A_dx=\sum_i P_{j_i}A(Jx)D_i.
\]
Since $(P_{j_i})_i$ and $(D_i)_i$ are separately pairwise orthogonal,
\eqref{eq:orthogonal-projections}, applied to $A(Jx)$, gives convergence
in ideal norm and $\norm{A_d}\leq M$.
For $i>s$,
\[
 A_de_{is}=(Ae_{j_i j_s})D_i\in P_{j_i}C_E(H)D_i,
\]
and \eqref{pg:one-sided-old-head}--\eqref{pg:one-sided-tail} give
\[
 \abs{(R((Ae_{j_i j_s})D_i))_{j_i j_s}}
       \geq\delta-\delta/4-\delta/4=\delta/2.
\]
The bound \eqref{pg:one-sided-old-head} uses $h_{j_s}\leq\delta/8$,
which controls every projection index, including those chosen after $j_s$.

The families $(P_{j_i}H)_i$ and $(D_iH)_i$ consist of mutually
orthogonal finite-dimensional spaces. Choose positive integers
$b_i\geq\max(\operatorname{rank}P_{j_i},\operatorname{rank}D_i)$ and put
\[
 Y=\overline{\bigoplus_{i\geq1}M_{b_i}(\K)}^{\,C_E}.
\]
Separate Hilbert isometries embed $P_{j_i}H$ and $D_iH$ into $\K^{b_i}$.
Their direct sums preserve singular values on the closed span of
$\{P_{j_i}C_E(H)D_i:i\geq1\}$. Reverse compression from each
$M_{b_i}(\K)$ to the images of $P_{j_i}H$ and $D_iH$, followed by the
inverse Hilbert isometries, defines a contraction $Y\to C_E(H)$. Compose $A_d$ with the insertion into $Y$, and $R$ with
this contraction and the extraction $C$. We obtain maps
$X_-(F)\to Y\to C_F$ of norms at most $M,r$; the value at $e_{is}$,
$i>s$, lies in $M_{b_i}(\K)$, and \eqref{pg:block-row-returns} holds with
$\delta/2$ in place of $\delta$. Proposition~\ref{pg:block-row} gives
the second alternative with the stated bounds.

If $t_E<\infty$, Proposition~\ref{pg:lower-matched} applies to the first
alternative. Its Hilbert possibility is excluded by the assumption on $F$,
so it gives canonical $E=F$.
Finally, for a right-supported component put
$A^{\mathsf t}x=(Ax)^{\mathsf t}$ and
$R^{\mathsf t}y=R(y^{\mathsf t})$. Then
$R^{\mathsf t}A^{\mathsf t}=RA$, and the left supports of
$A^{\mathsf t}e_{ij}$ are $P_i^{\mathsf t}$, still mutually orthogonal and
finite dimensional. The source indices and the coefficients in
\eqref{pg:one-sided-hypotheses} are unchanged, proving the last assertion.
\end{proof}

\begin{proposition}\label{pg:bounded-alternative}
Suppose \(t_E,t_F<\infty\), and let
\(T:C_F\to C_E\) and \(S:C_E\to C_F\) be bounded maps with \(ST=I\).
Then at least one of the following holds:
\[
 E=F,\qquad F=\ell_2,\qquad
 C_F\text{ embeds complementably into }S_E.
\]
\end{proposition}

\begin{proof}
Assume $F\ne\ell_2$. Apply Lemma~\ref{pg:full-reproduction} with
$Y=C_E$, $H_N=P_N$, and then Lemma~\ref{pg:coarse-components}.
For an infinite $I\subset\mathbb N$, one of
$C_{11},C_{12},C_{21},C_{10},C_{01}$ satisfies
$\abs{(RC_{ab}(e_{ij}))_{ij}}\geq1/5$ whenever $i>j$ and $i,j\in I$.
Restrict to matrix coordinates with both indices in $I$, and compose
$R$ with principal coordinate extraction along $I$.

For $i>j$, $C_{12}(e_{ij})\in q_iC_Eq_j$, where $q_i$ are the
finite-rank projections from Lemma~\ref{pg:coarse-components}.
Proposition~\ref{pg:lower-matched} gives $E=F$ or $F=\ell_2$, and
the latter is excluded. Applying $y\mapsto y^{\mathsf t}$ to
$C_{21}(e_{ij})$ gives left support $q_i^{\mathsf t}$ and right
support $q_j^{\mathsf t}$, so the same proposition applies.

For $i>j$, $C_{11}(e_{ij})\in q_iC_Eq_i$. Proposition~\ref{pg:block-row} therefore gives
$C_F\hookrightarrow S_E$ complementably.

For $i>j$, $C_{10}(e_{ij})=q_iC_{10}(e_{ij})$.
Proposition~\ref{pg:one-sided} gives either $E=F$ or a complemented
embedding $C_F\to S_E$. For $C_{01}$, apply the same proposition to
$e_{ij}\mapsto C_{01}(e_{ij})^{\mathsf t}$, whose left support is
$q_i^{\mathsf t}$, and to the return map $y\mapsto R(y^{\mathsf t})$.
The indices $i,j$ and the coefficients in \eqref{pg:return-sum} are preserved.
\end{proof}

\begin{proof}[Proof of Theorem~\ref{pg:uniqueness}]
Let \(U:C_F\to C_E\) be an isomorphism onto \(C_E\).
Lemma~\ref{pg:triangle-invariance}, applied to \(U\) and its
inverse, shows that \(t_E<\infty\) if and only if
\(t_F<\infty\). If these norms are infinite,
Theorem~\ref{pg:unbounded-uniqueness} gives \(E=F\).

Suppose that they are finite. Apply
Proposition~\ref{pg:bounded-alternative} with \(T=U\) and
\(S=U^{-1}\). The alternative that \(C_F\) embeds
complementably into \(S_E\) forces
\(F=\ell_2\) by Theorem~\ref{pg:comp-block}. Thus either
\(E=F\) or \(F=\ell_2\).

In the latter case, $C_E\simeq C_2$, so $E=\ell_2=F$;
see \cite[proof of Corollary 3.4]{Arazy}. Indeed, restrict a Hilbert-space
isomorphism $C_E\to H$ to the diagonal copy of $E$ and denote the restriction by $B$.
Symmetry and averaging over independent signs give, for $a\in c_{00}$,
\[
 \norm{a}_E^2\asymp
 \E\norm{\sum_i\varepsilon_i a_iBe_i}_H^2
 =\sum_i\abs{a_i}^2\norm{Be_i}_H^2\asymp\norm{a}_2^2,
\]
where the constants depend only on $B$ and its inverse on its range.
\end{proof}

\section{On Arazy's embedding problem}\label{sec:ordinary-embeddings}

\subsection{A positive answer under finite cotype}\label{sec:older-comparisons}

A Banach lattice of finite cotype is \(q\)-concave for some finite
\(q\ge2\), not necessarily equal to its cotype exponent
\cite[Theorems 4.4(b) and 4.6(iv)]{Maligranda}. Hence, for some
\(M_q<\infty\),
\begin{equation}\label{eq:concavity}
 \left(\sum_\nu\norm{u_\nu}_E^q\right)^{1/q}
 \le M_q\norm{\left(\sum_\nu\abs{u_\nu}^q\right)^{1/q}}_E .
\end{equation}
For a finite probability distribution, apply this inequality to the
appropriately weighted vectors \(u_\nu\).

We also use the quasi-Banach Kahane--Khintchine inequality in its
\(L_q/L_1\) form. The \(r\)-Banach formulation and the
Aoki--Rolewicz renorming are stated in
\cite[Theorem 2.16 and Section 2.1, arXiv:1804.08947v2]{CCV}.
We obtain a dimension-independent constant by applying these results
to the underlying real space of a single quasi-Banach space
containing all finite matrix spaces.

\begin{lemma}\label{lem:quasi}
For every finite \(q\ge1\), there is \(K_q<\infty\) such that, for every
\(d\ge1\), \(1\le j\le d\), and finite family \(a_i\in M_d(\K)\),
\[
 \begin{gathered}
 \left(\E Q_j\left(\sum_i r_i a_i\right)^q\right)^{1/q}
 \le K_q\,\E Q_j\left(\sum_i r_i a_i\right),\\
 Q_j(a)=\max_{1\le l\le j}(l/j)^2s_l(a).
 \end{gathered}
\]
Each Rademacher family consists of independent signs. The constant
is independent of \(d,j\) and the number of summands.
\end{lemma}

\begin{proof}
The function \(Q_j\) is homogeneous and positive definite, and
\(j^{-2}\norm{a}\le Q_j(a)\le\norm{a}\).
For \(1\le l\le j\), put \(h=\lceil l/2\rceil\). The singular-value
inequality and monotonicity give
\[
 s_l(a+b)\le s_h(a)+s_h(b),
 \qquad
 (l/j)^2s_l(a+b)\le4\bigl(Q_j(a)+Q_j(b)\bigr).
\]
The singular-value inequality bounds \(s_{2h-1}(a+b)\), and
\(2h-1\le l\). Thus the quasi-triangle constant is at most \(4\)
for every \(d,j\).

Let \(X\) be the completion, in the supremum quasi-norm,
of the algebraic direct sum of the spaces \((M_d(\K),Q_j)\),
where \(d\ge1\) and \(1\le j\le d\). Equivalently, \(X\) consists
of the families whose coordinate quasi-norms tend to zero. This is a
separable quasi-Banach space with quasi-triangle constant at most \(4\).
Completeness follows by taking the limit in each finite-dimensional
coordinate of a Cauchy family, then using the uniform Cauchy bound;
the vanishing-tail condition is preserved by approximation.
Each \((M_d,Q_j)\) is a linear isometric subspace of \(X\). Applying
Aoki--Rolewicz renorming and the quasi-Banach Kahane--Khintchine
inequality to \(X\), then returning to the original quasi-norm, gives
a constant \(K_q\) valid on every coordinate.
\end{proof}

For a fixed finite partition of coordinates into blocks of lengths
\(d_1,\ldots,d_N\), define the positive linear operator
\[
 (D_b^{(t)}u)_{k,j}=u_{k,\lceil j/t\rceil},
 \quad 1\le j\le d_k,\qquad t\in\N.
\]
For each \(k\), \(D_b^{(t)}\) repeats the coordinates
\(u_{k,1},\ldots,u_{k,d_k}\) \(t\) times and retains the first
\(d_k\) entries of the resulting list. For every symmetric Banach
sequence norm,
\begin{equation}\label{eq:dilation}
 \norm{D_b^{(t)}u}_E\le t\norm{u}_E.
\end{equation}
Repeating every coordinate of \(u\) \(t\) times without truncation
gives a sum of \(t\) disjoint copies of \(u\), each of the same
\(E\)-norm. The truncated block vector is a coordinate restriction
of a permutation of this vector.
Consequently the positive operator
\begin{equation}\label{eq:Lb}
 L_b=\sum_{h=0}^{\infty}2^{-2h}D_b^{(2^{h+1})}
 \quad\hbox{satisfies}\quad
 \norm{L_b}_{E\to E}\le
 \sum_{h=0}^{\infty}2^{-2h}2^{h+1}=4.
\end{equation}
The series converges in operator norm, so no Fatou assumption is needed.
For a decreasing nonnegative vector \(u_{k,1}\ge\cdots\ge u_{k,d_k}\),
\begin{equation}\label{eq:majorant}
 \max_{1\le l\le j}(l/j)^2u_{k,l}
 \le (L_bu)_{k,j}.
\end{equation}
For each \(l\le j\), choose \(h\ge0\) with
\(j/2^{h+1}<l\le j/2^h\). The term on the left is bounded by
\(2^{-2h}u_{k,\lceil j/2^{h+1}\rceil}\), one of the terms on the right.

\begin{lemma}\label{lem:block}
Suppose \(E\) satisfies \eqref{eq:concavity}. For every finite family of
mutually orthogonal finite-rank projections \(q_k\), and every array
\(x_{ik}=q_kx_{ik}q_k\), put
\[
 Y=\sum_i r_i\sum_k x_{ik},
 \qquad Z=\sum_{i,k}r_{ik}x_{ik}.
\]
There exists \(D_E<\infty\), independent of all dimensions and of the
array, such that
\[
 D_E^{-1}(\E\norm{Z}_{C_E}^2)^{1/2}
 \le(\E\norm{Y}_{C_E}^2)^{1/2}
 \le D_E(\E\norm{Z}_{C_E}^2)^{1/2}.
\]
One may take \(D_E=4M_qK_q\).
\end{lemma}

\begin{proof}
Discard any zero projections and put \(d_k=\operatorname{rank}q_k\).
Let \(W\) denote either \(Y\) or \(Z\), put \(W_k=q_kWq_k\), and
let \(u(W)=(u_{k,j}(W))_{k,1\le j\le d_k}\), where
\(u_{k,j}(W)=s_j(W_k)\). Thus \(u(W)\) concatenates the decreasing
singular-value lists of \(q_kWq_k\), and
\(\norm{W}_{C_E}=\norm{u(W)}_E\). The vector
\(v=\E u(W)\) is independent of the choice \(W=Y\) or \(W=Z\),
because \(q_kYq_k\) and \(q_kZq_k\) have equal distributions for every \(k\).

For each coordinate \((k,j)\), Lemma~\ref{lem:quasi} and
\eqref{eq:majorant} give
\[
 \bigl(\E u_{k,j}(W)^q\bigr)^{1/q}
 \le K_q\,\E Q_j(W_k)
 \le K_q(L_bv)_{k,j}.
\]
Since \(q\ge2\), \(q\)-concavity now yields
\begin{align*}
 (\E\norm{W}_{C_E}^2)^{1/2}
 &\le(\E\norm{u(W)}_E^q)^{1/q}\\
 &\le M_q\norm{(\E\abs{u(W)}^q)^{1/q}}_E\\
 &\le M_qK_q\norm{L_bv}_E
 \le4M_qK_q\norm{v}_E .
\end{align*}
Conversely, the triangle inequality for the finite expectation gives
\(\norm{v}_E\le\E\norm{u(W)}_E\le(\E\norm{W}_{C_E}^2)^{1/2}\).
These bounds apply to both random sums.
\end{proof}

Let \(T:C_F\to S_E\) be an isomorphic embedding. Normalize
\(\norm{e_1}_F=1\), and choose \(a,b>0\) with
\[
 a\norm{x}_{C_F}\le\norm{Tx}_{C_E}\le b\norm{x}_{C_F}.
\]
For $y=(y_k)_{k\ge1}\in S_E$, put
$H_Ny=(y_1,\ldots,y_N,0,\ldots)$ and $H_0=0$.
These finite-rank projections are nested, and
\eqref{eq:orthogonal-projections} gives $\norm{H_N}\leq1$ and $H_Ny\to y$
in $S_E$. Fix $0<\eta<a$ and apply Lemma~\ref{pg:full-reproduction}
with $Y=S_E$. Let $V$ and $0=N_0<N_1<\cdots$ be the resulting embedding
and integers, and set $Q_j=H_{N_j}-H_{N_{j-1}}$, so that $Q_jV=V\Pi_j$.
For $m,k\ge1$, define
\begin{equation}\label{eq:model-columns}
 U_k:\K^m\to S_E,\qquad
 U_k(v)=V\left(\sum_{i=1}^m v_i e_{i,m+k}\right).
\end{equation}
The input in \eqref{eq:model-columns} belongs to $\Pi_{m+k}C_F$,
so $U_k(\K^m)\subset Q_{m+k}S_E$. For fixed $m$, these ranges are
supported on both sides by the mutually orthogonal finite-rank
Hilbert projections $\sum_{t=N_{m+k-1}+1}^{N_{m+k}}q_t$.
For $v_1,\ldots,v_N\in\K^m$, the matrix
$\sum_{k=1}^N\sum_{i=1}^m(v_k)_i e_{i,m+k}$ has the same nonzero
singular values as $[v_1\ \cdots\ v_N]$.
Thus Lemma~\ref{pg:full-reproduction} gives, with
$A_0=a-\eta$ and $B_0=2b+\eta$,
\begin{equation}\label{eq:model}
 A_0\norm{[v_1\ \cdots\ v_N]}_{C_F}
 \le\norm{\sum_{k=1}^N U_kv_k}_{C_E}
 \le B_0\norm{[v_1\ \cdots\ v_N]}_{C_F}.
\end{equation}
The maps \(U_k:\K^m\to S_E\) depend on \(m\), but
\(A_0,B_0\) do not depend on \(m,N\).

\begin{proof}[Proof of Theorem~\ref{thm:main}]
Choose \(q,M_q\) as in \eqref{eq:concavity}. Fix \(N\) and
\(\alpha=(\alpha_1,\ldots,\alpha_N)\in\K^N\). For each \(m\), use the
maps in \eqref{eq:model-columns} and set
\(x_{ik}=\alpha_kU_k(e_i)\).
For \(Y=\sum_i r_i\sum_k\alpha_kU_k(e_i)\), the matrix in
\eqref{eq:model} is \([r_i\alpha_k]_{i\le m,k\le N}\). It has rank at most
one, and its only possible nonzero singular value is \(\sqrt m\norm{\alpha}_2\).
For \(Z=\sum_{i,k}r_{ik}\alpha_kU_k(e_i)\), the matrix in
\eqref{eq:model} is \(G_m\diag(\alpha)\), where
\(G_m=(r_{ik})_{i\le m,k\le N}\).
Thus \eqref{eq:model} and Lemma~\ref{lem:block} imply
\[
 A_0\sqrt m\norm{\alpha}_2\le D_EB_0 L_m(\alpha),\qquad
 A_0 L_m(\alpha)\le D_EB_0\sqrt m\norm{\alpha}_2,
\]
where
\[
 L_m(\alpha)=\bigl(\E\norm{G_m\diag(\alpha)}_{C_F}^2\bigr)^{1/2}.
\]
Couple \(G_m\) as the first \(m\) rows of one infinite Bernoulli
matrix. With \(N\) fixed, the strong law for the finitely many entries
of the Gram matrix gives \(m^{-1}G_m^*G_m\to I_N\) almost surely.
Hence the singular-value list of \(m^{-1/2}G_m\diag(\alpha)\),
padded to length \(N\), converges almost surely to \((\abs{\alpha_k})^*\).
The finite-dimensional restriction of the \(F\)-norm is continuous.
For all \(m\) and all outcomes,
\[
 m^{-1/2}\norm{G_m\diag(\alpha)}_{C_F}
 \le m^{-1/2}\norm{G_m\diag(\alpha)}_{C_1}
 \le\sqrt N\norm{\alpha}_2.
\]
The first bound follows from \(\norm{e_1}_F=1\); the second follows from
rank at most \(N\) and Hilbert--Schmidt norm exactly
\(\sqrt m\norm{\alpha}_2\). Dominated convergence of the squares proves
\[
 \lim_{m\to\infty}m^{-1/2}L_m(\alpha)=\norm{\alpha}_F.
\]
Therefore, for every finite sequence \(\alpha\),
\[
 \frac{A_0}{D_EB_0}\norm{\alpha}_2
 \le\norm{\alpha}_F
 \le\frac{D_EB_0}{A_0}\norm{\alpha}_2 .
\]
The constants are independent of \(N\). Since finite sequences are
dense in \(F\), the identity on \(c_{00}\) extends to an isomorphism
\(F\to\ell_2\). Continuity of the coordinate functionals shows that
this extension fixes every coordinate. Hence \(F=\ell_2\) as sets,
with equivalent norms.
With \(\eta=a/2\), the equivalence constant is at most
\(D_E(4b/a+1)\).
\end{proof}

Before the fixed-\(N\) Gram matrix limit, all expectations are on
finite probability spaces. Neither complementability nor the Fatou
property is used.

Under additional convexity assumptions, Lemma~\ref{lem:block} follows
from known Khintchine inequalities. Suppose that the
symmetric sequence norm is \(2\)-convex and
\(q\)-concave for some finite \(q\), allowing equivalent symmetric
renorming. The convexity and concavity constants can both be made one,
as recalled immediately before Theorem 1.1 in \cite{LPX}.
Set \(a_i=\sum_k x_{ik}\), with the notation of
Lemma~\ref{lem:block}. Two-sided disjointness gives
\[
 \sum_i a_i^*a_i=\sum_{i,k}x_{ik}^*x_{ik},\qquad
 \sum_i a_ia_i^*=\sum_{i,k}x_{ik}x_{ik}^*.
\]
The maximum row--column formula of \cite[Theorem 1.3(ii)]{LPX}
compares both \( (\E\norm{Y}_{C_E}^2)^{1/2}\) and
\((\E\norm{Z}_{C_E}^2)^{1/2}\) with
\[
 \max\left\{
 \norm{(\sum_{i,k}x_{ik}^*x_{ik})^{1/2}}_{C_E},
 \norm{(\sum_{i,k}x_{ik}x_{ik}^*)^{1/2}}_{C_E}
 \right\},
\]
proving Lemma~\ref{lem:block} under the stated \(2\)-convexity and
finite \(q\)-concavity assumptions. Equivalent
symmetric norms give equivalent ideal norms, so the comparison
transfers to the original norm with dimension-independent constants.

More generally, suppose a symmetric function space \(G\) on
\((0,\infty)\) has the product/Bochner comparison
\begin{equation}\label{eq:older-product}
 \norm{\sum_i r_i\otimes a_i}_{G(L_\infty(\Omega)
                  \mathbin{\overline\otimes}\mathcal A)}
 \asymp_G \E\norm{\sum_i r_i a_i}_{G(\mathcal A)},
\end{equation}
uniformly in finite families, and its norm on integer-step
singular-value functions agrees, up to fixed constants, with the
sequence norm of \(E\). Here \(\mathcal A\) is the direct sum of
matrix algebras, equipped with the sum \(\tau\) of their ordinary
traces. With \(d_W(t)=(\mathbb P\otimes\tau)
(\mathbf1_{(t,\infty)}(\abs{W}))\), blockwise marginal equality gives
\[
 d_Y(t)=\sum_k\E\operatorname{Tr}_{q_k}
  \mathbf1_{(t,\infty)}\left(\abs{\sum_i r_i x_{ik}}\right)
 =d_Z(t)\qquad(t>0).
\]
Thus \(Y\) and \(Z\) have equal product symmetric norms. Applying
\eqref{eq:older-product} to \((a_i)_i\) and \((x_{ik})_{i,k}\),
and then the Banach-space Kahane--Khintchine inequality, proves
the block comparison for \(E\). A fixed ambient algebra containing all
finite blocks gives constants independent of the array and its
dimensions.

The comparison \eqref{eq:older-product} is considered in
\cite[Section 4]{LMS}. Corollary 4.3 there covers the
\(2\)-concave and the \(2\)-convex, finite-concavity function-space
cases; the paragraph following it attributes the reflexive Orlicz
case to \cite{DS}. For an Orlicz function \(M\), the passage to
the sequence norm is exact, since
\[
 \tau\bigl(M(\abs{a}/\lambda)\bigr)
   =\sum_j M(s_j(a)/\lambda)
\]
for finite matrices. These results give Lemma~\ref{lem:block} in the indicated ranges.
The proof of Lemma~\ref{lem:block} uses finite cotype and does not require
\eqref{eq:older-product}.

\subsection{Examples and counterexamples}
\label{sec:ordinary-limits}\label{sec:comparison}

The embedding theorems of Lindenstrauss and Szankowski give a
counterexample to \cite[Question 1.1]{HSS} for unrestricted separable
symmetric parameters.

\begin{theorem}
\label{thm:counterexample}
Let \(F\) be a separable reflexive symmetric Banach sequence space
with non-trivial Boyd indices. Then there exists a separable reflexive
symmetric Banach sequence space \(E\) such that
\[
 C_F\hookrightarrow E\hookrightarrow S_E.
\]
\end{theorem}

\begin{proof}
Let \(C_F(\mathbb R)\) denote the real Banach space of matrices
in \(C_F\) with real entries. Since \(F\) is reflexive, \(C_F\)
and hence \(C_F(\mathbb R)\) are reflexive \cite{Arazy}.
The non-trivial Boyd indices of \(F\) make
\((\Pi_nC_F)_{n\ge1}\), with \(\Pi_n\) defined in \eqref{pg:shells},
an unconditional finite-dimensional decomposition
\cite[Corollary 4.6]{Arazy78}.
The projections \(\Pi_n\) preserve real matrices, so their estimates
restrict to \(C_F(\mathbb R)\).
By \cite[Theorem 1.g.5]{LT}, \(C_F(\mathbb R)\) embeds into a space
with an unconditional basis. Its image is a reflexive subspace,
so \cite[Theorem 3.1]{FJT} gives a reflexive space \(X\) with an
unconditional basis into which \(C_F(\mathbb R)\) embeds.
By Szankowski's theorem \cite{Szankowski}, in the strengthened form of
\cite[Theorem 3.3]{DKM}, \(X\) embeds into a reflexive space \(Y\)
with a normalized \(1\)-symmetric basis \((y_n)\).
Its coefficient space \(E_{\mathbb R}=\{a:\sum_n a_ny_n\text{ converges in }Y\}\), with
\(\norm{a}_{E_{\mathbb R}}=\norm{\sum_n a_ny_n}_Y\), is a separable reflexive
symmetric Banach sequence space containing \(C_F(\mathbb R)\).

Set \(E=\{z:\abs{z}\in E_{\mathbb R}\}\), with
\(\norm{z}_E=\norm{\abs{z}}_{E_{\mathbb R}}\).
For \(u,v\in E_{\mathbb R}\), the inequalities
\(\max\{\norm{u},\norm{v}\}\le\norm{u+iv}\le\norm{u}+\norm{v}\)
show that this space is again separable and reflexive.
If \(T:C_F(\mathbb R)\to E_{\mathbb R}\) has embedding bounds \(a,b>0\),
then \(T_{\mathbb C}(A)=T(\operatorname{Re}A)+iT(\operatorname{Im}A)\),
with real and imaginary parts taken entrywise,
is complex linear and has bounds \(a/2,2b\): the same inequalities hold
in \(C_F\), since entrywise conjugation preserves its norm.
Thus \(C_F\hookrightarrow E\).

Finally, define
\begin{equation}\label{eq:diagonal}
 D(a)=\bigoplus_{n\ge1}a_ne_{11}^{(n)}\qquad(a\in E),
\end{equation}
where \(e_{11}^{(n)}\in M_n\). This is the isometric embedding from
\cite[Proposition 1.1]{Arazy81Geometry}, using the first coordinate of
each summand $\K^n$. Its range lies in $S_E$, so
$\norm{D(a)}_{S_E}=\norm{a}_E$.
\end{proof}

\begin{remark}\label{rem:counterexample-scope}
The host \(E\) depends on \(F\). These classical theorems do not
assert complementability of the displayed copy of \(C_F\), or
an isomorphism \(C_E\simeq C_F\) or \(C_E\simeq S_E\).
If \(F\ne\ell_2\), such an \(E\) has no finite cotype by
Theorem~\ref{thm:main}.
\end{remark}

For a fixed separable symmetric space \(E\), write \(\mathcal P(E)\)
for the assertion that, for every separable symmetric sequence space \(F\),
\begin{equation}\label{eq:PE}
 \mathcal P(E):\qquad
 C_F\hookrightarrow S_E\ \Longrightarrow\ F=\ell_2 .
\end{equation}
Finite cotype implies \(\mathcal P(E)\) by Theorem~\ref{thm:main},
but it is not necessary, as the \(E=c_0\) case of the nonembedding
theorem of Huang, Sadovskaya and Sukochev
\cite[Theorem 4.1]{HSS} shows.

\begin{proposition}\label{prop:c0}
The space \(c_0\) has no finite cotype and satisfies
\(\mathcal P(c_0)\). In fact, no \(C_F\) embeds into \(S_{c_0}\).
\end{proposition}

\section*{Acknowledgements}
This work was supported by the NNSF of China
(Nos. 12031004, 12301160, 12471134 and 12671159).

\end{document}